\documentclass[11pt]{amsart}
\usepackage{amssymb}

\usepackage{amsfonts, amsmath, amssymb, amsgen, amsthm, amscd,latexsym,mathrsfs}
\usepackage{setspace}
\usepackage{color}
\usepackage[all]{xy}
\usepackage[top=3.2cm, left=2.2cm, bottom=3cm, right=2.2cm]{geometry}
\usepackage[T1]{fontenc}

\def\Ad{\mathop{\rm Ad}\nolimits}
\def\ad{\mathop{\rm ad}\nolimits}

\def\Cb{{\mathbb C}}

\def\Rb{{\mathbb R}}

\def\b{\beta}
\def\d{\delta}

\def\s{\sigma}

\def\0b{\bf 0}

\def\nb{\nabla}
\def\ot{\otimes}

\def\ra{\rightarrow}
\def\lra{\longrightarrow}

\def\rt{\triangleright}

\def\lt{\triangleleft}
\def\cl{\blacktriangleright\hspace{-4pt} < }
\def\crr{>\hspace{-4pt}\blacktriangleleft}

\def\dcc{\blacktriangleright\hspace{-4pt}\blacktriangleleft }

\def\0D{\Delta^{(0)}}
\def\1D{\Delta^{(1)}}

\newcommand{\Fg}{\mathfrak{g}}
\newcommand{\Fh}{\mathfrak{h}}

\newtheorem{theorem}{Theorem}[section]

\newtheorem{proposition}[theorem]{Proposition}
\newtheorem{lemma}[theorem]{Lemma}
\newtheorem{corollary}[theorem]{Corollary}

\def\build#1_#2^#3{\mathrel{
\mathop{\kern 0pt#1}\limits_{#2}^{#3}}}

\newcommand{\nsb}[1]{~\hspace{-4pt}_{^{[#1]}}}

\def\b{\beta}
\def\d{\delta}

\def\s{\sigma}

\def\nb{\nabla}

\def\ot{\otimes}
\def\part{\partial}

\def\ra{\rightarrow}

\def\lra{\leftrightarrow}
\def\text{\hbox}

\def\nb{\nabla}
\def\ot{\otimes}

\def\ra{\rightarrow}

\def\Ad{\mathop{\rm Ad}\nolimits}

\def\lra{\longrightarrow}

\def\build#1_#2^#3{\mathrel{
\mathop{\kern 0pt#1}\limits_{#2}^{#3}}}

\numberwithin{equation}{section}

\email{oesen@gtu.edu.tr}
\email{serkan.sutlu@isikun.edu.tr}

\begin{document}
\title{Hamiltonian Dynamics on Matched Pairs}
\author{O\u{g}ul Esen}
\address{Gebze Technical University, 41400, Gebze-Kocaeli, Turkey}
\author{Serkan S\"utl\"u}
\address{I\c{s}{\i}k University, 34980, \c{S}ile-\.Istanbul, Turkey}
\maketitle

\begin{abstract}
The cotangent bundle of a matched pair Lie group, and its trivialization, are shown to be a matched pair Lie group. The explicit matched pair decomposition on the trivialized bundle is presented. On the trivialized space, the canonical symplectic two-form and the canonical Poisson bracket are explicitly written. Various symplectic and Poisson reductions are perfomed. The Lie-Poisson bracket is derived. As an example, Lie-Poisson equations on $\mathfrak{sl}(2,\mathbb{C})^\ast$ are obtained.

\end{abstract}

\section{Introduction}

A matched pair Lie group $G\bowtie H$ is a Lie group by itself containing $G$
and $H$ as two non-intersecting Lie subgroups acting on each other subject to certain compatibility conditions \cite{LuWein90,Maji90,Maji90-II,Majid-book,Take81}. In our previous work \cite{essu2016LagMP}, we have studied the Lagrangian dynamics on the matched pair Lie algebra $\Fg\bowtie \Fh$, and the tangent bundle $T(G\bowtie H)$. In the present work, we investigate the geometry of the cotangent bundle $T^\ast(G\bowtie H)$, and present the symplectic structure, together with the Hamiltonian dynamics on it. Moreover, we consider various symplectic and
Poisson reductions, and  obtain the Lie-Poisson equations on the matched pair Lie coalgebra  $\Fg^\ast\dcc \Fh^\ast$.

The dynamics on matched pairs are interesting from many points of views. The existence of the mutual group actions leads to a nontrivial, and purely geometrical, way of coupling two dynamical equations of two interacting systems. Conversely, realizing the configuration space of a dynamical system as a matched pair, one can decompose the configuration space, as well as the dynamical equations, into simpler parts. Furthermore, if one of the mutual actions is trivial, then the matched pair Lie group reduces to a semi-direct product Lie group. In that sense,
the dynamics on matched pairs can be understood as a generalization of the well-studied semi-direct product theory, \cite{esen2015tulczyjew,marsden1998symplectic, marsden1991symplectic,marsden1984reduction,marsden1984semidirect,Rati80}.

The paper is organized in three sections.

The following section is devoted to the basic definitions related to the matched pairs, and the lifts of the group actions to
tangent and cotangent spaces. It is this section that we discuss the matched pair decomposition of the cotangent group $T^{\ast}(G\bowtie H)$, and we observe that it is the matched pair group of the semi-direct products $(\Fh^{\ast}\rtimes G)$ and $(\Fg^{\ast}\rtimes H)$.

In the third section, we study the symplectic structure on $T^\ast(G\bowtie H)$, and we derive the Hamilton's equations. In this section we also consider the
symplectic and Poisson reductions with respect to the symmetries given by $G$, $H$, and $G\bowtie H$. In particular, under the reflexivity conditions, a reduced Hamiltonian function(al) $\mathcal{H}=\mathcal{H}(\mu,\nu)$ on $\Fg^{\ast}\dcc\Fh^{\ast}$ generates, what we call, \emph{the matched Lie-Poisson equations}
\begin{equation}
\begin{cases}
\label{mLP1} \frac{d\mu}{dt} =ad^{\ast}_{\frac{\delta\mathcal{H}}{\delta\mu}%
}(\mu)+\frac{\delta\mathcal{H}}{\delta\nu} \overset{\ast }{\vartriangleright%
}\mu+ \mathfrak{a}_{\frac{\delta\mathcal{H}}{\delta\nu}}^{\ast}\nu, \\
\frac{d\nu}{dt} =ad^{\ast}_{ \frac{\delta\mathcal{H}}{\delta\nu}}(\nu)+%
\frac{\delta\mathcal{H}}{\delta\mu} \overset{\ast }{\vartriangleright}\nu+
\mathfrak{b}_{\frac{\delta\mathcal{H}}{\delta\mu}}^{\ast}\mu.%
\end{cases}%
\end{equation}
In the set \eqref{mLP1}, the first terms in the right hand sides come
from the individual Lie-Poisson equations for the dual spaces $\Fg^{\ast}$
and $\Fh^{\ast}$, respectively. The second and the third terms of \eqref{mLP1} on the right
however, are the results of the existence of the infinitesimal actions of $\Fg$ on $\Fh^\ast$,
and the actions of $\Fh$ on $\Fg^\ast$ induced from the group actions.
We also show in this section that if, in particular, one of these actions is trivial, then \eqref{mLP1} reduces to
the semidirect product Lie-Poisson dynamics. We conclude the section by the presentations of
the Legendre transformation, under the nondegeneracy conditions,
from the Lagrangian dynamics on $T(G\bowtie H)$ and $\Fg\bowtie\Fh$,
to the Hamiltonian dynamics on $T^\ast (G\bowtie H)$ and $\Fg^\ast \dcc \Fh^\ast$, respectively.

The forth, and the last, section is reserved to study the Lie-Poisson equations (\ref{mLP1}) on a concrete example. Just as we did in our previous work \cite{essu2016LagMP}, we consider  $SL(2,\mathbb{C})$ and its Iwasawa decomposition.

For the sake of completeness, and in order to widen the spectrum of the potential readers, we include five appendices. In the first one, we briefly summarize the symplectic and the Poisson reductions in a more abstract framework \cite{marsden1986reduction, marsden1974reduction,meyer1973symmetries}. The second appendix is for the Hamiltonian dynamics on Lie groups, whereas the third appendix recalls the Hamiltonian reduction from \cite{MarsdenRatiu-book}. On the fourth appendix we provide the basics on the Legendre transformation for a dynamical system whose configuration space is a Lie group.
In the last appendix, the definition of a Lie coalgebra is given \cite{Majid-book,Mich80}.

\section{Cotangent Bundle of a Matched Pair Group}

In this section we summarize the notion of a matched pair of Lie groups, and illustrate the group structure of the cotangent group of a matched pair Lie group.

Let us first fix our notation. Let $G$ and $H$ be two Lie groups, with the Lie algebras $\Fg$ and $\Fh$, respectively. We
shall denote the elements by
\begin{align*}
&g, g_{1}, g_{2} \in G ,\qquad h, h_{1}, h_{2} \in H ,\qquad \xi, \xi_{1},
\xi_{2} \in\Fg ,\qquad \eta, \eta_{1}, \eta_{2} \in\Fh,\qquad U_{g}\in T_{g}G
,\qquad V_{h}\in T_{h}H, \\
&\hspace{2cm}\mu, \hat{\mu},\mu_{1}, \mu_{2} \in\Fg^{\ast},\qquad \nu, \hat{\nu},\nu_{1},
\nu_{2} \in\Fh^{\ast} ,\qquad \alpha_{g}\in T_{g}^{\ast}G ,\qquad
\beta_{h}\in T_{h}^{\ast}H,
\end{align*}
where $TG$, $TH$ are the tangent bundles, and $T ^\ast G$, $T ^\ast H$ are the cotangent bundles of the groups $G$ and $H$, respectively.
The Lie group $G$ acts on its Lie algebra $\Fg$ from the left by the
adjoint action $\xi\mapsto\Ad_g\xi$, and hence the Lie algebra $\Fg$ acts on itself
from the left by the infinitesimal adjoint action $\xi_2\mapsto\ad_{\xi_1}\xi_2$. The (left) coadjoint action $\mu\mapsto\Ad^{\ast}_{g}\mu$ of $G$ on the dual space $\Fg^{\ast}$ is then given by
\begin{equation}  \label{Ad*}
\left\langle \Ad^{\ast}_{g}\mu,\xi \right\rangle= \left\langle \mu,\Ad%
_{g^{-1}}\xi \right\rangle.
\end{equation}
On the infinitesimal level, the (left) coadjoint action $\mu\mapsto\ad^{\ast}_\xi\mu$ of $\Fg$ on its dual $\Fg^{\ast}$ is given by
\begin{equation}  \label{ad*}
\left\langle \ad^{\ast}_{\xi_{1}}\mu,\xi_{2} \right\rangle=-\left\langle \mu,%
\ad_{\xi_{1}}{\xi_{2}} \right\rangle.
\end{equation}
%In this notation, we have that
%\begin{equation*}
%\frac{d}{dt}\bigg|_{t=0} \Ad^{\ast}_{g^{t}}\mu=ad^{\ast}_{\xi}\mu
%\end{equation*}
%when $g^{t}$ is a curve passing through the identity element of $G$ for $t=0$
%in the direction of $\xi$.

\subsection{Matched pair of Lie groups, Lie algebras and Lie coalgebras}
In this subsection we recall the basics on the matched pairs of Lie groups, Lie algebras and Lie coalgebras. We refer the reader to \cite{essu2016LagMP,LuWein90,Maji90-II,Maji90,Majid-book,Mich80,Take81,zhang2010double} for further details on the subject.

Let $(G,H)$ be a pair of Lie groups, such that $H$ acts on $G$ from the left,
and $G$ acts on $H$ from the right by
\begin{equation}
\rho :H\times G\rightarrow G,\qquad \left( h,g\right) \mapsto
h\vartriangleright g, \qquad \sigma :H\times G\rightarrow H,\qquad \left(
h,g\right) \mapsto h\vartriangleleft g.  \label{rho}
\end{equation}
The pair $(G,H)$ is called a matched pair of Lie groups if the mutual actions
\eqref{rho} satisfy the compatibility conditions
\begin{align*}
\begin{array}{rcl}
h\vartriangleright \left( g_{1}g_{2}\right) =\left( h\vartriangleright
g_{1}\right) \left( \left( h\vartriangleleft g_{1}\right) \vartriangleright
g_{2}\right), && (h_{1}h_{2})\vartriangleleft g=\left(
h_{1}\vartriangleleft \left( h_{2}\vartriangleright g\right) \right) \left(
h_{2}\vartriangleleft g\right), \\
h\vartriangleright e_G = e_G, && e_H\vartriangleleft g = e_H,
\end{array}
\end{align*}
where $e_G \in G$ and $e_H\in H$ are the identity elements. In this case, the Cartesian product $G\times H$ becomes a Lie group with the multiplication
\begin{equation*}
\left( g_{1},h_{1}\right) \left(
g_{2},h_{2}\right) =\left( g_{1}\left( h_{1}\vartriangleright
g_{2}\right) ,\left( h_{1}\vartriangleleft g_{2}\right) h_{2}\right) =\left(
g_{1}\rho \left( h_{1},g_{2}\right) ,\sigma \left( h_{1},g_{2}\right)
h_{2}\right),
\end{equation*}
and is denoted by $G\bowtie H$. In case one of the actions in \eqref{rho} is trivial, the matched pair group
$G\bowtie H$ reduces to a semi-direct product.

We next recall the matched pair of Lie algebras. Let $(\Fg,\Fh)$ be a pair of
Lie algebras equipped with the mutual actions we denote by
\begin{equation} \label{Lieact}
\vartriangleright:\Fh\ot\Fg\rightarrow \Fg,\quad \eta\ot \xi \mapsto \eta\vartriangleright \xi, \hspace{3cm} \vartriangleleft:\Fh\ot\Fg\rightarrow \Fh, \quad \eta\ot \xi \mapsto \eta\vartriangleleft \xi.
\end{equation}
Then the pair $(\Fg,\Fh)$ is called a matched pair of Lie algebras if the compatibilities
\begin{eqnarray*}
\eta \vartriangleright \lbrack \xi _{1},\xi _{2}]=[\eta \vartriangleright
\xi _{1},\xi _{2}]+[\xi _{1},\eta \vartriangleright \xi _{2}]+(\eta
\vartriangleleft \xi _{1})\vartriangleright \xi _{2}-(\eta \vartriangleleft
\xi _{2})\vartriangleright \xi _{1}, \\
\lbrack \eta _{1},\eta _{2}]\vartriangleleft\xi =[\eta _{1},\eta _{2}\vartriangleleft\xi ]+[\eta _{1}%
\vartriangleleft\xi ,\eta _{2}]+\eta _{1}\vartriangleleft (\eta _{2}\vartriangleright \xi
)-\eta _{2}\vartriangleleft (\eta _{1}\vartriangleright \xi ),
\end{eqnarray*}
are satisfied. In this case, the sum $\Fg \oplus \Fh$ is a Lie algebra with the Lie bracket
\begin{equation}
\lbrack (\xi _{1},\eta _{1}),\,(\xi _{2},\eta _{2})]=\left( [\xi _{1},\xi
_{2}]+\eta _{1}\vartriangleright \xi _{2}-\eta _{2}\vartriangleright \xi
_{1},\,[\eta _{1},\eta _{2}]+\eta _{1}\vartriangleleft \xi _{2}-\eta
_{2}\vartriangleleft \xi _{1}\right),  \label{mpla}
\end{equation}%
and is denoted by  $\mathfrak{g\bowtie h}$. We note that the Lie algebra of the matched pair group $G\bowtie H$ is the matched pair Lie algebra $\Fg\bowtie \Fh$, see for instance \cite{Maji90-II}.

We conclude the present subsection with a brief discussion on the matched pair of Lie coalgebras. We refer the reader to Appendix \ref{LCo} for details on Lie coalgebras, and their coactions. A pair $(\mathfrak{G},\mathfrak{H})$ of Lie coalgebras is called a matched pair of Lie coalgebras if $\mathfrak{G}$ is a left
$\mathfrak{H}$-comodule, and $\mathfrak{H}$ is a right $\mathfrak{G}$-comodule, such that
\begin{align*}
& \mu\nsb{-1} \ot \mu\nsb{0}\nsb{1} \ot \mu\nsb{0}\nsb{2} = \mu\nsb{1}%
\nsb{-1} \ot \mu\nsb{1}\nsb{0} \ot \mu\nsb{2} + \mu\nsb{2}\nsb{-1} \ot \mu%
\nsb{1} \ot \mu\nsb{2}\nsb{0} + \mu\nsb{-1}\nsb{0} \ot (\mu\nsb{-1}\nsb{1} %
\ot \mu\nsb{0} - \mu\nsb{0} \ot \mu\nsb{-1}\nsb{1}),  \notag \\
& \nu\nsb{0}\nsb{1} \ot \nu\nsb{0}\nsb{2} \ot \nu\nsb{1} =\nu\nsb{1} \ot \nu%
\nsb{2}\nsb{0} \ot \nu\nsb{2}\nsb{1} + \nu\nsb{1}\nsb{0} \ot \nu\nsb{2} \ot %
\nu\nsb{1}\nsb{1} + \nu\nsb{0} \ot \nu\nsb{1}\nsb{-1} \ot \nu\nsb{1}\nsb{0}
- \nu\nsb{1}\nsb{-1} \ot \nu\nsb{0} \ot \nu\nsb{1}\nsb{0},  \notag
\end{align*}
for any $\mu \in \mathfrak{G}$, and any $\nu\in \mathfrak{H}$, see for
instance \cite{zhang2010double}. In this case, $\mathfrak{G}\dcc\mathfrak{H}:=\mathfrak{%
G}\oplus\mathfrak{H}$ becomes a Lie coalgebra with the cobracket given by
\begin{align*}
\d_{\mathfrak{G}\dcc\mathfrak{H}}(\mu) = \mu\nsb{1}\ot \mu\nsb{2} + \mu\nsb{-1}
\ot \mu\nsb{0} - \mu\nsb{0}\ot \mu\nsb{1}, \qquad \d_{\mathfrak{G}\dcc%
\mathfrak{H}}(\nu) = \nu\nsb{1}\ot \nu\nsb{2} + \nu\nsb{0} \ot \nu\nsb{1} -
\nu\nsb{-1}\ot \nu\nsb{0}.
\end{align*}
Let $\Fg\bowtie\Fh$ be a matched pair Lie algebra. Then, dualizing the
mutual actions, $(\Fg\bowtie\Fh)^{\ast}$ becomes a Lie
coalgebra denoted by $\Fg^\ast\dcc\Fh^\ast$. Recalling (\ref{mpla}), the cobracket is given by
\begin{eqnarray} \label{Cob}
\Big\langle \d_{\Fg^\ast\dcc\Fh^\ast}(\mu,\nu), (\xi_1,\eta_1)\ot (\xi_2,\eta_2)\Big\rangle&=&\Big\langle
(\mu,\nu),[(\xi_1,\eta_1),(\xi_2,\eta_2)]\Big\rangle \nonumber \\
&=& \Big\langle \mu, [\xi _{1},\xi _{2}]+\eta _{1}\vartriangleright \xi
_{2}-\eta _{2}\vartriangleright \xi _{1}\Big\rangle+ \Big\langle \nu, [\eta
_{1},\eta _{2}]+\eta _{1}\vartriangleleft \xi _{2}-\eta _{2}\vartriangleleft
\xi _{1}\Big\rangle.
\end{eqnarray}
In the sequel we shall identify $(\Fg\bowtie\Fh)^{\ast}$ with $\Fg^\ast\dcc\Fh^\ast$. We therefore write, for the right trivialization of the cotangent bundle of a matched pair group, $T^{\ast}(G\bowtie H)\simeq (\Fg^\ast\dcc\Fh^\ast)\rtimes (G\bowtie H)$. As we shall see in the forthcoming section, the cobracket \eqref{Cob} determines the structure of the Lie-Poisson bracket, and hence the reduced dynamics, on $\Fg^\ast\dcc\Fh^\ast$.

\subsection{Tangent and cotangent lifts of the actions}

Freezing one of the arguments of the left action $\rho:H\times G\to G$ of \eqref{rho}, we arrive at two differentiable mappings
\begin{align} \label{sigma}
\rho_{g} & :H\rightarrow G,\quad h\mapsto h\vartriangleright g,\hspace{2.5cm}
\rho_{h}:G\rightarrow G,\quad g\mapsto h\vartriangleright g,
\end{align}
the tangent lifts of which are
\begin{equation} \label{tliftrho}
T_{h}\rho_{g}:T_{h}H\rightarrow T_{h\vartriangleright g}G,\hspace{2cm} T_{g}\rho
_{h}:T_{g}G\rightarrow T_{h\vartriangleright g}G,
\end{equation}
respectively. In particular, over the identity they become
\begin{align} \label{TeHrho}
T_{e_{H}}\rho_{g} & :\Fh\rightarrow T_{g}G\qquad\eta\mapsto
\eta\vartriangleright g, \qquad T_{e_{G}}\rho_{h} :\mathfrak{g}\rightarrow\mathfrak{g}\qquad\xi\mapsto
h\vartriangleright\xi.
\end{align}
Hence, the lifted action of the group $H$ on the tangent bundle $TG$, and the Lie algebra $\Fg$ are given by
\begin{eqnarray}
\hat{\rho}&:&H\times TG\rightarrow TG\qquad\left( h,U_{g}\right) \mapsto
h\vartriangleright U_{g}:=T_{g}\rho_{h}\left( U_{g}\right) \in
T_{h\vartriangleright g}G,  \label{rho-hat}
\\
\hat{\rho}&:&H\times\mathfrak{g}\rightarrow\mathfrak{g}\qquad\left(
h,\xi\right) \mapsto h\vartriangleright\xi:=T_{e}\rho_{h}\left( \xi\right) .
\label{rho-hat-algebra}
\end{eqnarray}
Differentiating (\ref{rho-hat}) with respect to the
group variable, we obtain the mappings
\begin{eqnarray}
T_{h}\hat{\rho}_{U_{g}}:T_{h}H\rightarrow T_{h\vartriangleright U_{g}}TG
\qquad V_{h}\mapsto T_{h}\hat{\rho}_{U_{g}}\left( V_{h}\right)
=:V_{h}\vartriangleright U_{g}.  \label{Trho-hat}
\end{eqnarray}
%where (\ref{Trho-tilda}) is the infinitesimal action of the group $TH$ on $TG $. Note that $T_{g}\tilde{\rho}_{V_{h}}\left( V_{g}\right) = T_{h}%
%\hat{\rho}_{V_{g}}\left( V_{h}\right)$. As a result, we use the same notation $V_{h}\vartriangleright U_{g}$ for both.
In particular, for $h=e_{H}$ and $g=e_{G}$, this mapping reduces to
\begin{equation}
\mathfrak{b}_\xi:=T_{e_{H}}%
\hat{\rho}_{\xi }:\Fh\rightarrow\mathfrak{g},\quad\eta\mapsto\eta
\vartriangleright\xi  \label{etaonxi}
\end{equation}
which enables us to define the left action $(\eta,\xi)\mapsto\eta
\vartriangleright\xi$ of $\Fh$ on $\Fg$ in \eqref{Lieact}. By freezing $\eta$ we arrive at the mapping
\begin{equation}  \label{etaonxi2-1}
\mathfrak{b}_\eta:\mathfrak{g}\rightarrow\mathfrak{g},\quad
\xi\mapsto\eta\vartriangleright\xi,
\end{equation}
which is the infinitesimal generator of the action \eqref{rho-hat-algebra}.

We next discuss the cotangent lifts briefly. The cotangent lift $T^\ast \rho_h$ gives the right action of the group $H$ on the cotangent bundle $T^\ast G$, and the linear algebraic dual $\Fg^\ast$ of $\Fg$ by
\begin{eqnarray}
\hat{\rho}^{\ast }&:&T^{\ast }G\times H\rightarrow T^{\ast }G\qquad \left(
\alpha _{g},h\right) \mapsto \alpha _{g}\overset{\ast }{\vartriangleleft }h:=%
{\hat{\rho}_{h}}^{\ast }(\alpha _{g})=T_{h^{-1}\vartriangleright g}^{\ast
}\rho _{h}\left( \alpha _{g}\right) \in T_{h^{-1}\vartriangleright g}^{\ast
}G,  \label{Trho-hat}
\\
\hat{\rho}^{\ast }&:&\mathfrak{g}^{\ast }\times H\rightarrow \mathfrak{g}%
^{\ast }\qquad \left( \mu ,h\right) \mapsto \mu \overset{\ast }{%
\vartriangleleft }h:=T_{e}^{\ast }\rho _{h}(\mu )
\end{eqnarray}%
where the dualization is
\begin{equation}\label{honmu}
\left\langle \alpha _{g}\overset{\ast }{\vartriangleleft }h, U_{h^{-1}\vartriangleright g}\right\rangle =\left\langle \alpha
_{g},h \vartriangleright U_{h^{-1}\vartriangleright g}\right\rangle.
\end{equation}
Finally, the dualizations of the mappings in (\ref{etaonxi}) and (\ref{etaonxi2-1}) are given by
\begin{equation}\label{b-dual}
\langle \mathfrak{b}_{\xi}^{\ast}\mu,\eta\rangle =\langle \mu, \mathfrak{b}%
_{\xi}\eta\rangle,\qquad \langle \mu\overset{\ast}{\vartriangleleft}\eta, \xi \rangle =\langle \mu,\eta\vartriangleright \xi \rangle,
\end{equation}
respectively. Note that, in these dualizations, we arrive at the infinitesimal right action of $\Fh$ on $\Fg^\ast$.

Similar arguments can be repeated for the right action $\sigma:H\times G \to H$. This time we start with the mappings
\begin{equation*}
\sigma_{h} :G\rightarrow H,\quad g\mapsto h\vartriangleleft g, \hspace{2.5cm}
\sigma_{g}:H\rightarrow H,\quad h\mapsto h\vartriangleleft g,
\end{equation*}
and their derivatives
\begin{eqnarray}
T_{g}\sigma _{h}&:&T_{g}G\rightarrow T_{h\vartriangleleft g}H,\hspace{2.5cm}
T_{h}\sigma _{g}:T_{h}H\rightarrow T_{h\vartriangleleft g}H,  \label{TeGsigma}
\\
\label{lifted-sigma}
T_{e_{G}}\sigma _{h}&:&\mathfrak{g}\rightarrow T_{h}H,\quad \xi \mapsto
h\vartriangleleft \xi ,\hspace{2cm}
T_{e_{H}}\sigma _{g}:\Fh\rightarrow \Fh,\quad \eta \mapsto
\eta \vartriangleleft g.
\end{eqnarray}
The right action of $G$ on the tangent bundle $TH$ and $\Fh$ are given by
\begin{align}
\hat{\sigma}& :TH\times G\rightarrow TH,\qquad \left( V_{h},g\right) \mapsto
V_{h}\vartriangleleft g:=T_{h}\sigma _{g}\left( V_{h}\right) \in
T_{h\vartriangleleft g}H,
\\
\hat{\sigma}& :\Fh\times G\rightarrow \Fh,\qquad \left(\eta,g\right) \mapsto
\eta\vartriangleleft g:=T_{e_H}\sigma _{g}\left( \eta\right) \in
\Fh. \label{Gonh}
\end{align}%
Freezing the vector entry, we arrive at $\hat{\sigma}_{V_{h}}:G\rightarrow TH$ whose derivative is
\begin{align*}
T_{g}\hat{\sigma}_{V_{h}}& :T_{g}G\rightarrow T_{V_{h}\vartriangleleft
g}TH\qquad U_{g}\mapsto V_{h}\vartriangleleft U_{g}.
\end{align*}
%Then, we obtain $T_{h}\tilde{\sigma}_{V_{g}}\left( V_{h}\right) =T_{g}\hat{%
%\sigma}_{V_{h}}\left( V_{g}\right) $, and employ the notation $%
%V_{h}\vartriangleleft U_{g}$ for both. Furthermore, in case
In particular, for $g=e_{G}$ and $h=e_{H}$, we obtain
\begin{equation}
\mathfrak{a}_{\eta
}:=T_{e_{G}}\hat{\sigma}_{\eta }:\mathfrak{g}\rightarrow \Fh,\quad
\xi \mapsto \eta \vartriangleleft \xi  \label{a}
\end{equation}
which yields the right action $(\eta,\xi)\mapsto\eta
\vartriangleleft\xi$ of $\Fg$ on $\Fh$ in (\ref{Lieact}).
By freezing $\xi$,
\begin{equation}  \label{etaonxi2}
\mathfrak{a}_\xi:\mathfrak{h}\rightarrow\mathfrak{g},\quad
\xi\mapsto\eta\vartriangleright\xi
\end{equation}
is the infinitesimal generator of the action \eqref{Gonh}.

%T_{e_{G}}\hat{\sigma}_{\eta }:\Fh\rightarrow \Fh,\quad
%\eta \mapsto \eta \vartriangleleft \xi ,\hspace{2cm}

As for the cotangent lifts, we have
\begin{eqnarray}
\hat{\sigma}^{\ast }&:&G\times T^{\ast }H\rightarrow T^{\ast }H,\hspace{2cm} \left(
g,\b _{h}\right) \mapsto g\overset{\ast }{\vartriangleright }\b
_{h}:=T_{h\vartriangleleft g^{-1}}^{\ast }\sigma _{g}\left( \b
_{h}\right) \in T_{h\vartriangleleft g^{-1}}^{\ast }H,\label{GonT*H}
\\ \label{gonh*}
\hat{\sigma}^{\ast }&:&G\times \Fh^{\ast }\rightarrow \Fh%
^{\ast },\hspace{2cm} \left( \nu ,g\right) \mapsto g\overset{\ast }{%
\vartriangleright }\nu :=T_{e_{H}}^{\ast }\sigma _{g}\left( \nu \right),
\end{eqnarray}
where
\begin{equation*}
\left\langle g\overset{\ast }{\vartriangleright }\b
_{h},V_{h\vartriangleleft g^{-1}}\right\rangle =\left\langle \b
_{h},V_{h\vartriangleleft g^{-1}}\vartriangleleft g\right\rangle.
\end{equation*}
In particular, we have the dual actions
\begin{eqnarray}\label{a-dual}
\left\langle \xi \overset{\ast }{\vartriangleright }\nu ,\eta
\right\rangle=\left\langle \nu ,\eta \vartriangleleft \xi \right\rangle,
\qquad \langle \mathfrak{a}_{\eta}^{\ast}\nu, \xi \rangle = \langle \nu ,
\mathfrak{a}_\eta(\xi) \rangle =\langle \nu , \eta \vartriangleleft \xi
\rangle,
\end{eqnarray}
where the mapping $\mathfrak{a}_\eta$ is in (\ref{a}).

We conclude with the (co)adjoint action of the matched pair group $G\bowtie H$ on the Lie algebra $\Fg\bowtie \Fh$, and its dual $(\Fg\bowtie \Fh)^\ast$. The tangent lifts of the left and right regular actions of $G\bowtie H$ are given by
\begin{align}
& T_{( g_{2},h_{2})}L_{( g_{1},h_{1})}( U_{g_{2}},V_{h_{2}}) = \left(
T_{h_{1}\vartriangleright g_{2}}L_{g_{1}}( h_{1}\vartriangleright U_{g_{2}})
,\,T_{h_{1}\vartriangleleft g_{2}}R_{h_{2}}\left( h_{1}\vartriangleleft
U_{g_{2}}\right) +T_{{h_{2}}}L_{( h_{1}\vartriangleleft g_{2})
}V_{h_{2}}\right),  \label{righttransl} \\
& T_{\left( g_{1},h_{1}\right) }R_{\left( g_{2},h_{2}\right) }\left(
U_{g_{1}},V_{h_{1}}\right) = \left( T_{g_{1}}R_{\left(
h_{1}\vartriangleright g_{2}\right) }U_{g_{1}}+T_{h_{1}\vartriangleright
g_{2}}L_{g_{1}}\left( V_{h_{1}}\vartriangleright g_{2}\right)
,T_{h_{1}\vartriangleleft g_{2}}R_{h_{2}}\left( V_{h_{1}}\vartriangleleft
g_{2}\right) \right),  \label{righttransl2}
\end{align}
where $(U_{g_{i}},V_{h_{i}})\in T_{(g_i,h_i)}(G \bowtie H)$ for $i=1,2$.
%The adjoint and coadjoint actions of the group $G\bowtie H$ to its Lie
%algebra and the dual space, respectively.
Following (\ref{Ad*}), for any $(g,h) \in G \bowtie H$, $(\xi,\eta)\in \mathfrak{g}%
\bowtie \Fh$, $(\mu,\nu)\in (\mathfrak{g}\bowtie \Fh%
)^{\ast}$, we have
\begin{eqnarray}
Ad_{(g,h)^{-1}}(\xi,\eta)=(h^{-1}\vartriangleright\zeta,T_{h^{-1}}R_{h}
(h^{-1}\vartriangleleft\zeta)+Ad_{h^{-1}}(\eta\vartriangleleft g)),
\label{Ad} \\
Ad_{(g,h)}^{\ast}(\mu,\nu)=(Ad^{\ast}_g\lambda, T^{\ast}_{e_H}\rho_g\circ
T^{\ast}L_{g^{-1}}\lambda+T^{\ast}_{e_H}\rho_g\circ Ad^{\ast}_h \nu)
\label{Ad-*}
\end{eqnarray}
where we use the abbreviations
\begin{equation*}
\zeta := Ad_{g^{-1}}\xi+T_gL_{g^{-1}}(\eta\vartriangleright g) \in \mathfrak{%
g}, \qquad \lambda=h\overset{\ast}{\vartriangleright}\mu+T^{\ast}_{e_G}%
\sigma_{h^{-1}}\circ T^{\ast}R_h\nu \in \mathfrak{%
g}^\ast.
\end{equation*}
%Here, $T^{\ast}_{e_H}\rho_g$ is the dual of the mapping $T_{e_H}\rho_g$ in (%
%\ref{TeHrho}), and $T^{\ast}_{e_G}\sigma_{h^{-1}}$ is the dual of $%
%T_{e_G}\sigma_{h^{-1}}$ in (\ref{TeGsigma}).
Accordingly, the infinitesimal adjoint
action of $\mathfrak{g}\bowtie \Fh$ on its dual space $(\mathfrak{g}%
\bowtie \Fh)^{\ast}$ is given by
\begin{equation}  \label{ad-*}
ad_{(\xi,\eta)}^{\ast}(\mu,\nu)= (ad^{\ast}_{\xi}(\mu)+\mu \overset{\ast }{%
\vartriangleleft}\eta+ \mathfrak{a}_{\eta}^{\ast}\nu,
ad^{\ast}_{\eta}(\nu)+\xi \overset{\ast }{\vartriangleright}\nu+ \mathfrak{b}%
_{\xi}^{\ast}\mu).
\end{equation}

\subsection{Cotangent bundle of a matched pair Lie group}

\label{CBMPLG}

The right trivialization
of the cotangent bundle of a Lie group yields
\begin{equation}  \label{iso}
T^\ast(G\bowtie H) \cong (\mathfrak{g}\bowtie \Fh)^\ast \rtimes
(G\bowtie H)= (\Fg^\ast\dcc \Fh^\ast)\rtimes
(G\bowtie H),
\end{equation}
by which we identify the group structure with that of the semi-direct product $(\Fg^\ast\dcc \Fh^\ast)\rtimes
(G\bowtie H)$. The group multiplication on the trivialized space is
\begin{eqnarray}  \label{GrT*GH}
\qquad (\mu_{1} ,\nu_{1} ,g_{1},h_{1})\cdot( \mu_{2} ,\nu_{2} ,g_{2},h_{2})=
((\mu_{1} ,\nu_{1})+Ad^{\ast}_{(g_{1},h_{1})} (\mu_{2}
,\nu_{2}),(g_{1}\left( h_{1}\vartriangleright g_{2}\right) ,\left(
h_{1}\vartriangleleft g_{2}\right) h_{2})),
\end{eqnarray}
where $Ad^{\ast}$ is the coadjoint action \eqref{Ad-*} of the matched pair group
$G\bowtie H$ on the dual space $\Fg^\ast\dcc \Fh^\ast$. A direct calculation shows that the
inclusions $T^\ast G \hookrightarrow T^\ast(G\bowtie H) \hookleftarrow
T^\ast H$ are not group homomorphisms. Instead, we have the following
decomposition.
\begin{lemma}
Let $G\bowtie H$ be a matched pair Lie group, and $(\Fg^\ast\dcc \Fh^\ast)\rtimes (G\bowtie H)$ be the trivialization of the matched pair cotangent group. Then the mappings
\begin{equation*}
\mathfrak{g}^\ast \rtimes H \hookrightarrow (\Fg^\ast\dcc \Fh^\ast) \rtimes (G\bowtie H) \hookleftarrow \Fh^\ast \rtimes G
\end{equation*}
given by
\begin{equation*}
(\mu,h) \hookrightarrow ((\mu,0),(e_G,h)), \qquad ((0,\nu),(g,e_H)) \hookleftarrow (\nu,g)
\end{equation*}
determine two subgroups.
\end{lemma}
\begin{proof}
Let us first note that $\mathfrak{g}^\ast \rtimes H$ is the semi direct product group formed by the action of the group $H$ on the Lie algebra
$\mathfrak{g}$, and hence to $\mathfrak{g}^\ast$. Similarly the group structure $\Fh%
^\ast \rtimes G$ is determined by the action of the group $G$ on
$\Fh^\ast$. It follows from
\begin{equation*}
Ad^\ast_{(e_G,h)}(\mu,0) = (h \overset{\ast}{\vartriangleright} \mu,0), \qquad
Ad^\ast_{(g,e_H)}(0,\nu) = (0,g \overset{\ast}{\vartriangleright} \nu)
\end{equation*}
that these actions indeed coincide with the coadjoint action of the group $%
G\bowtie H$ on its Lie algebra $\mathfrak{g}\bowtie \Fh$. We
finally see that the images of the embeddings are indeed closed, under
multiplication, by
\begin{equation*}
((\mu_1,0),(e_G,h_1)) ((\mu_2,0),(e_G,h_2)) = ((\mu_1 + h \overset{\ast}{\vartriangleright}
\mu_2,0),(e_G,h_1h_2)),
\end{equation*}
and similarly by
\begin{equation*}
((0,\nu_1),(g_1,e_H)) ((0,\nu_2),(g_2,e_H)) = ((0,\nu_1 + g_1\overset{\ast}{\vartriangleright}
\nu_2),(g_1g_2,e_H)).
\end{equation*}
\end{proof}
\begin{lemma}
The mapping
\begin{equation}\label{multp}
(\mathfrak{g}^\ast \rtimes H) \times (\Fh^\ast \rtimes G) \lra (\Fg^\ast\dcc \Fh^\ast)\rtimes (G\bowtie H)
\end{equation}
given by the multiplication in $(\Fg^\ast\dcc \Fh^\ast)\rtimes (G\bowtie H)$ is an isomorphism.
\end{lemma}

\begin{proof}
Let us first note that
\begin{align*}
(\mu,h) \cdot (\nu,g)& = ((\mu,0),(e_G,h))((0,\nu),(g,e_H)) \\
& = ((\mu,0) + Ad^\ast_{(e_G,h)}(0,\nu),(h\rt g,h\lt g)) \\
& =((\mu + \nu \circ T_{h^{-1}}R_h \circ T_{e_G}\s_{h^{-1}},
Ad^\ast_h(\nu)),(h\rt g,h\lt g)).
\end{align*}
Next, given $((\widehat{\mu},\widehat{\nu}),(\widehat{g},\widehat{h})) \in (\Fg^\ast\dcc \Fh^\ast)\rtimes (G\bowtie H)$, we see that $h %
\rt g = \widehat{g}$ and $h\lt g = \widehat{h}$ implies $g = (\widehat{h}%
^{-1} \rt \widehat{g}^{-1})^{-1}$ and $h = (\widehat{h}^{-1} \lt \widehat{g}%
^{-1})^{-1}$, respectively. As a result,
\begin{align*}
& ((\widehat{\mu},\widehat{\nu}),(\widehat{g},\widehat{h})) \mapsto \\
& ((\widehat{\mu} - Ad^\ast_{\widehat{h}^{-1} \lt \widehat{g}^{-1}}(\widehat{%
\nu}) \circ T_{\widehat{h}^{-1} \lt \widehat{g}^{-1}}R_{(\widehat{h}^{-1} %
\lt \widehat{g}^{-1})^{-1}} \circ T_{e_G}\s_{\widehat{h}^{-1} \lt \widehat{g}%
^{-1}}, (\widehat{h}^{-1} \lt \widehat{g}^{-1})^{-1}),(Ad^\ast_{\widehat{h}%
^{-1} \lt \widehat{g}^{-1}}(\widehat{\nu}), (\widehat{h}^{-1} \rt \widehat{g}%
^{-1})^{-1}))  \notag
\end{align*}
is the inverse of the mapping \eqref{multp}.
\end{proof}

\begin{corollary} \label{alt}
Given a matched pair $(G,H)$ of Lie groups,
\begin{equation*}
(\Fg^\ast\dcc \Fh^\ast) \rtimes (G\bowtie H) \cong (\mathfrak{g}^\ast \rtimes H) \bowtie (\Fh^\ast \rtimes G).
\end{equation*}
Furthermore, the mutual actions are given by
\begin{eqnarray*}
&(\nu,g) \rt (\mu,h) = (Ad^\ast_g(\mu) - Ad^\ast_{h^{-1} \lt g^{-1}}(\nu + \mu\circ TL_{g^{-1}} \circ T\rho_g) \circ TR_{(h^{-1} \lt g^{-1})^{-1}} \circ T\s_{h^{-1} \lt g^{-1}},(h^{-1}\lt g^{-1})^{-1}), \\ &
(\nu,g) \lt (\mu,h) = (Ad^\ast_{h^{-1} \lt g^{-1}}(\nu + \mu\circ TL_{g^{-1}} \circ T\rho_g),(h^{-1}\rt g^{-1})^{-1}).
\end{eqnarray*}
Accordingly, the product on the trivialized cotangent group $(\mathfrak{g}^\ast \rtimes H) \bowtie (\Fh^\ast \rtimes G)$
is given by
\begin{align*}
((\mu_1,h_1),(\nu_1,g_1)) \cdot ((\mu_2,h_2),(\nu_2,g_2)) =
((\mu_1,h_1)((\nu_1,g_1) \rt (\mu_2,h_2)),((\nu_1,g_1) \lt %
(\mu_2,h_2))(\nu_2,g_2)).
\end{align*}
\end{corollary}

\section{Hamiltonian Dynamics on Matched Pairs}

\subsection{Symplectic structure and Hamilton's equations}

Let us first note that the Lie algebra of the right trivialized cotangent bundle $(\Fg^\ast\dcc\Fh^\ast)\rtimes (G\bowtie H)$ is the quadrable $(\Fg^\ast\dcc\Fh^\ast)\rtimes \left( \Fg\bowtie\Fh\right)$. The choice right trivialization is motivated by the goal of expressing the right invariant vector fields (those generating the left action) in a concise form. Following Appendix \ref{reduc-Hamilton-dynam}, a right invariant vector field $X$ on $(\Fg^\ast\dcc\Fh^\ast)\rtimes (G\bowtie H)$ is generated by an
element $\left(\hat{\mu},\hat{\nu},\xi,\eta\right) \in (\Fg%
^\ast\dcc\Fh^\ast)\rtimes \left( \Fg\bowtie\Fh\right) $ as
\begin{align}  \label{RIVF}
& X_{\left(\hat{\mu},\hat{\nu},\xi,\eta\right) }\left(\mu ,\nu, g,h\right)
=\left((\hat{\mu},\hat{\nu})+ad_{(\xi,\eta)}^{\ast }(\mu ,\nu),TR_{(g,h)}{%
(\xi,\eta)}\right),
\end{align}
where $ad^{\ast }$ is the infinitesimal coadjoint action in (\ref{ad-*}), and $TR$ the tangent lift of the right translation on
$G\bowtie H$ given by \eqref{righttransl2}. In view of (\ref{1formT*G}) and (\ref{2formT*G}), the values of the pull-backs of
the canonical one-form $\theta$, and hence the symplectic two-form $\Omega$ on $(\Fg^\ast\dcc \Fh^\ast)\rtimes (G\bowtie H)$, on the right invariant
vector fields are
\begin{align}
\theta(X_{(\hat{\mu},\hat{\nu},\xi,\eta)})\left(\mu ,\nu,
g,h\right)&=\left\langle \mu,\xi\right\rangle +\left\langle \nu
,\eta\right\rangle, \\
\Omega(X_{(\hat{\mu_1},\hat{\nu_1},\xi_1,\eta_1)},X_{(\hat{\mu_2},\hat{\nu_2}%
,\xi_2,\eta_2)})\left(\mu ,\nu, g,h\right)&= \left\langle (\hat{\mu_2},\hat{%
\nu_2}),(\xi_1,\eta_1)\right\rangle \label{SymT*GH}\\
&-\left\langle (\hat{\mu_1},\hat{\nu_1}),(\xi_2,\eta_2)\right\rangle
+\left\langle (\mu ,\nu) ,[(\xi_1,\eta_1),(\xi_2,\eta_2)]\right\rangle, \nonumber
\end{align}
where the bracket on the right hand side is the one in (\ref{mpla}).

\begin{proposition}
For a Hamiltonian functional $\mathcal{H}=\mathcal{H}(\mu,\nu,g,h)$ on $(\Fg^\ast\dcc \Fh^\ast)\rtimes (G\bowtie H)$, the Hamilton's equations are
\begin{equation}
\begin{cases}
\label{MPHamEq}
\displaystyle
\frac{dg}{dt} =TR_{g}\left(\frac{\delta\mathcal{H}}{\delta\mu}\right)+\frac{\delta\mathcal{H}}{\delta\nu}\vartriangleright g, \\
\displaystyle \frac{dh}{dt} =TR_{h}\left(\frac{\delta\mathcal{H}}{\delta\nu}\vartriangleleft g\right), \\
\displaystyle\frac{d\mu}{dt} =-T_{e_G}^{\ast}R_g\left(\frac{\delta\mathcal{H}}{\delta g}\right) + ad^{\ast}_{\frac{\delta\mathcal{H}}{\delta\mu}}(\mu) +
\mu\overset{\ast }{\vartriangleleft}\frac{\delta\mathcal{H}}{\delta\nu} + \mathfrak{a}_{\frac{\delta\mathcal{H}}{\delta\nu}}^{\ast}\nu, \\
\displaystyle \frac{d\nu}{dt} =-T^{\ast}_{e_H}\rho_g\left(\frac{\delta\mathcal{H}}{\delta g}\right) -
g\overset{\ast}{\vartriangleright}T_{e_H}^{\ast}R_h\left(\frac{\delta\mathcal{H}}{\delta h}\right) + ad^{\ast}_{\frac{\delta\mathcal{H}}{\delta\nu}}(\nu) +\frac{\delta\mathcal{H}}{\delta\mu} \overset{\ast }{\vartriangleright}\nu + \mathfrak{b}_{\frac{\delta\mathcal{H}}{\delta\mu}}^{\ast}\mu.
\end{cases}
\end{equation}
Here $\vartriangleright$  denotes the infinitesimal left action \eqref{TeHrho} of $\Fh$ on $G$, $\vartriangleleft$ is the lifted right action \eqref{lifted-sigma} of $G$ on $\Fh$. The action $\overset{\ast}{\vartriangleright}$ stands both for the dual left action \eqref{gonh*} of $G$ on $\Fh^\ast$, and \eqref{a-dual} of $\Fg$ on $\Fh^\ast$, whereas $\overset{\ast}{\vartriangleleft}$ is for the right action $\Fh$ on $\Fg^\ast$ given by \eqref{b-dual}. Finally, the dual mappings $\mathfrak{a}^\ast$ and $\mathfrak{b}^\ast$ are given by \eqref{a-dual} and \eqref{b-dual} respectively.
\end{proposition}

\begin{proof}
Recall that for a Hamiltonian functional $\mathcal{H}$, the Hamilton's equations are implicitly given by
\begin{equation}  \label{HameqonT*GH}
\frac{d}{dt}(\mu,\nu)=-T^{\ast}R_{(g,h)}(\frac{\delta \mathcal{H}}{\delta
(g,h)})+ad^{\ast}_{\frac{\delta \mathcal{H}}{\delta (\mu,\nu)}}(\mu,\nu),
\qquad \frac{d}{dt}(g,h)=TR_{(g,h)}(\frac{\delta \mathcal{H}}{\delta
(\mu,\nu)}),
\end{equation}
where $T^{\ast}R$ is the cotangent lift of the right translation on $G\bowtie H$. Dualizing \eqref{righttransl2}, it is given by
\begin{equation}
T_{(e_G,e_H)}^{\ast}R_{(g,h)}(\alpha_g,\beta_h)=\Big(T_{e_G}^{\ast}R_g(%
\alpha_g),T^{\ast}_{e_H}\rho_g(\alpha_g) +g\overset{\ast}{\vartriangleright}%
T_{e_H}^{\ast}R_h(\beta_h)\Big).
\end{equation}
The claim now follows from a direct computation.
\end{proof}

\begin{proposition}
The canonical Poisson bracket corresponding to the symplectic structure \eqref{SymT*GH} is

\begin{align}  \label{PoissononT*GH}
\left\{ \mathcal{H},\mathcal{F}\right\} _{T^{\ast}(G\bowtie
H)}(\mu,\nu,g,h) =\underbrace{\left\langle T^{\ast}R_{g}\left(\frac{\delta
\mathcal{H}}{\delta g}\right),\frac{\delta \mathcal{F}}{\delta \mu}\right\rangle
-\left\langle T^{\ast}R_{g}\left(\frac{\delta \mathcal{F}}{\delta g}\right),\frac{%
\delta \mathcal{H}}{\delta \mu}\right\rangle +\left\langle \mu ,\left[\frac{%
\delta \mathcal{H}}{\delta \mu},\frac{\delta \mathcal{F}}{\delta \mu}%
\right]\right\rangle}_{K:\text{ Canonical Poisson bracket on $\Fg^{\ast}\rtimes G$}%
} \\
\underbrace{+\left\langle T^{\ast}\rho_{g}\left(\frac{\delta \mathcal{H}}{\delta
g}\right),\frac{\delta \mathcal{F}}{\delta \nu}\right\rangle -\left\langle
T^{\ast}\rho_{g}\left(\frac{\delta \mathcal{F}}{\delta g}\right),\frac{\delta
\mathcal{H}}{\delta \nu}\right\rangle +\left\langle \nu ,\left[\frac{\delta
\mathcal{H}}{\delta \nu},\frac{\delta \mathcal{F}}{\delta \nu}%
\right]\right\rangle}_{L:\text{ Poisson bracket on $\Fh^{\ast}\rtimes G$}}  \notag
\\
\underbrace{ +\left\langle g\overset{\ast}{\vartriangleright}%
T_{e_H}^{\ast}R_h\left(\frac{\delta \mathcal{H}}{\delta h}\right) ,\frac{\delta
\mathcal{F}}{\delta \nu} \right\rangle -\left\langle g\overset{\ast}{%
\vartriangleright}T_{e_H}^{\ast}R_h\left(\frac{\delta \mathcal{F}}{\delta h}\right),%
\frac{\delta \mathcal{H}}{\delta \nu}\right\rangle}_{M_1:\text{ Terms by
the action of $G$ on $T^{\ast}H$}} \underbrace{ -\left\langle \frac{%
\delta \mathcal{H}}{\delta \mu}\overset{\ast}{\vartriangleright}\nu ,\frac{%
\delta \mathcal{F}}{\delta \nu} \right\rangle -\left\langle \frac{\delta
\mathcal{F}}{\delta \mu}\overset{\ast}{\vartriangleright} \nu,\frac{\delta
\mathcal{H}}{\delta \nu}\right\rangle}_{M_2:\text{ Terms by the action
of $\Fg$ on $\Fh^\ast$}}  \notag \\
\underbrace{+ \left\langle \mu ,\frac{\delta \mathcal{H}}{\delta \nu}%
\vartriangleright \frac{\delta \mathcal{F}}{\delta \mu}\right\rangle
-\left\langle \mu ,\frac{\delta \mathcal{F}}{\delta \nu}\vartriangleright
\frac{\delta \mathcal{H}}{\delta \mu}\right\rangle }_{N:\text{ Terms by
the action of $\Fh$ on $\Fg^\ast$}}.  \notag
\end{align}
\end{proposition}
The proof of this can be deduced directly by the proper substitutions in (\ref{PoissononT*G}). We classify all terms on the right hand side of (\ref{PoissononT*GH}) with respect to the actions. The terms
labeled by $K$ are those obtained by the canonical Poisson bracket on $\Fg^{\ast}\ltimes
G$, see for instance (\ref{PoissononT*G}). Those labeled by $M_1$ and $M_2$
are the results of the action (\ref{GonT*H}) of $G$ on $T^{\ast}H$ and its
infinitesimal generator, respectively. Finally, the terms labeled by $N$ 
depend on the action (\ref{TeHrho}) of $\Fh$ on $\Fg$. Similarly, the ones labeled by $L$ are obtained from the
Poisson bracket on $\Fh^{\ast}\ltimes G$ which is, in view of the matched pair decomposition $(\mathfrak{g}^\ast \rtimes H)
\bowtie (\Fh^\ast \rtimes G)$ of the total space $(\Fg^\ast\dcc\Fh%
^\ast)\rtimes (G\bowtie H)$,
\begin{eqnarray}  \label{Poissononh*G}
\left\{ \mathcal{H},\mathcal{F}\right\} _{\Fh^{\ast}\rtimes G}(\nu,g)
=\left\langle T^{\ast}\rho_{g}(\frac{\delta \mathcal{H}}{\delta g}),\frac{%
\delta \mathcal{F}}{\delta \nu}\right\rangle -\left\langle T^{\ast}\rho_{g}(%
\frac{\delta \mathcal{F}}{\delta g}),\frac{\delta \mathcal{H}}{\delta \nu}%
\right\rangle +\left\langle \nu ,[\frac{\delta \mathcal{H}}{\delta \nu},%
\frac{\delta \mathcal{F}}{\delta \nu}]\right\rangle
\end{eqnarray}
for two functions $\mathcal{H}=\mathcal{H}(\nu,g)$ and $\mathcal{F}=\mathcal{F}(\nu,g)$. Comparing (\ref%
{Poissononh*G}) with the canonical Poisson bracket (\ref{PoissononT*G}) on $%
\Fg^{\ast}\rtimes G$, we observe that the roles of dual spaces $\Fh^{\ast}
$ and $ \Fg^{\ast}$, and the cotangent lifts $T^{\ast}\rho_g%
$ and $  T^{\ast}R_g$ are interchanged. Note that, the Poisson
bracket (\ref{Poissononh*G}) is non-degenerate if the action $\rho$ is free
and transitive.

\subsection{Lie-Poisson reduction}
It follows from the discussion in Appendix \ref{reduc-Hamilton-dynam} that the action
of $G \bowtie H$ on the trivialized space $(\Fg^\ast\dcc\Fh^\ast)\rtimes
(G\bowtie H)$ is symplectic.

\begin{proposition}
The Lie-Poisson bracket of two functions $\mathcal{H}= \mathcal{H}(\mu,\nu)$ and $\mathcal{F}=\mathcal{F}(\mu,\nu)$ on $\Fg^\ast\dcc\Fh^\ast$ is given by
\begin{align}  \label{LiePoissonongh}
& \left\{ \mathcal{H},\mathcal{F}\right\} _{\Fg^\ast\dcc\Fh%
^\ast}(\mu,\nu) =\underbrace{\left\langle \mu ,\left[\frac{\delta \mathcal{H}}{%
\delta \mu},\frac{\delta \mathcal{F}}{\delta \mu}\right]\right\rangle
+\left\langle \nu ,\left[\frac{\delta \mathcal{H}}{\delta \nu},\frac{\delta
\mathcal{F}}{\delta \nu}\right]\right\rangle}_{\text{A: direct product}} +\\
& \hspace{2cm} \underbrace{\left\langle \mu ,\frac{\delta \mathcal{H}}{\delta \nu}%
\vartriangleright \frac{\delta \mathcal{F}}{\delta \mu}\right\rangle
-\left\langle \mu ,\frac{\delta \mathcal{F}}{\delta \nu}\vartriangleright
\frac{\delta \mathcal{H}}{\delta \mu}\right\rangle}_ {\text{B: via the
action $\rho$}} +\underbrace{ \left\langle \nu ,\frac{\delta \mathcal{H}}{\delta \nu}%
\vartriangleleft \frac{\delta \mathcal{F}}{\delta \mu}\right\rangle
-\left\langle \nu ,\frac{\delta \mathcal{F}}{\delta \nu}\vartriangleleft
\frac{\delta \mathcal{H}}{\delta \mu}\right\rangle} _{\text{C: via the
action $\sigma$}}.  \notag
\end{align}
\end{proposition}

\begin{proof}
The Poisson reduction of the symplectic
manifold $(\Fg^\ast\dcc\Fh^\ast)\rtimes (G\bowtie H)$ with respect to the action of $G\bowtie H$ yields the
Lie-Poisson bracket on $\Fg^\ast\dcc\Fh^\ast$. Alternatively, it can be defined directly by
\begin{equation*}
 \left\{ \mathcal{H},\mathcal{F}\right\} _{\Fg^\ast\dcc\Fh%
^\ast}(\mu,\nu) = \Big\langle \d_{\Fg^\ast\dcc\Fh^\ast}(\mu,\nu), \left(\frac{\delta \mathcal{H}}{%
\delta \mu},\frac{\delta \mathcal{H}}{%
\delta \nu}\right)\ot \left(\frac{\delta \mathcal{F}}{%
\delta \mu},\frac{\delta \mathcal{F}}{%
\delta \nu}\right)\Big\rangle.
\end{equation*}
Hence, the claim follows from the Lie coalgebra structure \eqref{Cob}.
\end{proof}

We call \eqref{LiePoissonongh} the matched Lie-Poisson
bracket. The terms labeled by A correspond to the Lie-Poisson brackets (\ref{LPB})
on $\Fg^{\ast}$ and $\Fh^{\ast}$, respectively. The third and fourth terms, labeled by B are the
manifestations of the left action of the group $H$ on $G$, and those labeled by C are the incarnations of the right action of $G$ on $H$.

We note that the Lie-Poison bracket on the direct product $(\Fg %
\times \Fh)^{\ast}=\Fg^{\ast} \times \Fh^{\ast}$, regarding the mutual actions trivial,
is represented by the terms labeled by A. Assuming only the right action of $G$ on $H$
to be trivial, the terms labeled by A and B encode the Lie-Poison bracket on the semi-direct
coproduct $(\Fg \rtimes \Fh)^{\ast}=\Fg^{\ast} \crr \Fh^{\ast}$. Finally, the terms
labeled by A and C together form the Lie-Poison bracket on the semi-direct
coproduct $(\Fg \ltimes \Fh)^{\ast}=\Fg^{\ast} \cl \Fh^{\ast}$. We refer the
reader to \cite{marsden1998symplectic,marsden1984reduction,marsden1984semidirect,marsden1983coadjoint,
marsden1983hamiltonian,
ortega2013momentum} for the semi-direct product theories.

\begin{corollary}
The Lie-Poisson equations on $\Fg^\ast\dcc\Fh^\ast$ are
\begin{equation}
\begin{cases}
\label{LPEgh} \displaystyle\frac{d\mu}{dt} =ad^{\ast}_{\frac{\delta\mathfrak{%
H}}{\delta\mu}}(\mu)+\mu\overset{\ast }{\vartriangleleft}\frac{%
\delta\mathcal{H}}{\delta\nu} + \mathfrak{a}_{\frac{\delta\mathcal{H}}{%
\delta\nu}}^{\ast}\nu, \\
\displaystyle\frac{d\nu}{dt} =ad^{\ast}_{\frac{\delta\mathcal{H}}{\delta\nu}%
}(\nu)+\frac{\delta\mathcal{H}}{\delta\mu} \overset{\ast }{%
\vartriangleright}\nu+ \mathfrak{b}_{\frac{\delta\mathcal{H}}{\delta\mu}%
}^{\ast}\mu.%
\end{cases}%
\end{equation}
\end{corollary}
%\begin{equation}
%\begin{cases}
%\label{LPEgh} \displaystyle\frac{d\mu}{dt} =ad^{\ast}_{\frac{\delta\mathfrak{%
%H}}{\delta\mu}}(\mu)+\underbrace{\mu\overset{\ast }{\vartriangleleft}\frac{%
%\delta\mathcal{H}}{\delta\nu} + \mathfrak{a}_{\frac{\delta\mathcal{H}}{%
%\delta\nu}}^{\ast}\nu}_{\text{B: Due to the
%action $\rho$}} \\
%\displaystyle\frac{d\nu}{dt} =ad^{\ast}_{\frac{\delta\mathcal{H}}{\delta\nu}%
%}(\nu)+\underbrace{\frac{\delta\mathcal{H}}{\delta\mu} \overset{\ast }{%
%\vartriangleright}\nu+ \mathfrak{b}_{\frac{\delta\mathcal{H}}{\delta\mu}%
%}^{\ast}\mu}_{\text{C: Due to the
%action $\sigma$}}.%
%\end{cases}%
%\end{equation}
%that we call the matched Lie-Poisson equations on $\Fg^\ast\dcc\Fh^\ast$.
%Note that, the
%first terms on the right hand sides are comes from the individual
%Lie-Poisson equations (\ref{LPB}), and the second and the third terms in %
%\eqref{LPEgh} manifest the existence of the infinitesimal actions and their
%dualizations induced from the group actions. When, the action $\sigma$ is
%trivial, we have $H\ltimes G$ and the terms labeled by C drop in the set (%
%\ref{LPEgh}), whereas for the trivial $\rho$, we have $H\rtimes G$ and the
%terms labeled by B drop in the set (\ref{LPEgh}).

\subsection{Poisson Reduction by the actions of $G$ and $H$}

The groups $G$ and $H$ act on their matched pair $G\bowtie H$, and these actions can be lifted to the dual space $%
\Fg^\ast\dcc\Fh^\ast$, and to the cotangent bundle $T^{\ast }(G\bowtie H)$.
In view of \eqref{iso}, we arrive at the actions on the trivialized
space $(\Fg^\ast\dcc\Fh^\ast)\rtimes (G\bowtie H)$. We note that this is equivalent to the group multiplication (\ref{GrT*GH}) on $(\Fg^\ast\dcc\Fh^\ast)\rtimes (G\bowtie H)$, regarding $G$ and $H$ as subgroups. We further remark that these actions of $%
G$ and $H$ on $(\Fg^\ast\dcc\Fh^\ast)\rtimes (G\bowtie H)$ are symplectic.
Poisson reduction of the total space $(\Fg^\ast\dcc\Fh^\ast)\rtimes
(G\bowtie H)$ by the right action of the group $H$ results with the
semi-direct product Poisson manifold $(\Fg^\ast\dcc\Fh^\ast)\rtimes G$ with
a Poisson bracket which is structurally the same as (\ref{PoissononT*GH}) without the terms
labeled by $M_1$. The Poisson reduction of $(\Fg^\ast\dcc\Fh%
^\ast)\rtimes G$ under the right action of $G$ results with the Lie-Poisson bracket (\ref{LPEgh}) on $\Fg^\ast\dcc%
\Fh^\ast$, \cite{marsden2007hamiltonian}.

\subsection{Symplectic Reduction by $G$}

The symplectic action of the group $G$ on the trivialized total space $(\Fg^\ast\dcc%
\Fh^\ast)\rtimes (G\bowtie H)$ is
generated by the right invariant vector fields $X_{(0,0,\xi,0)}$, taking $%
\hat{\mu},\hat{\nu},\eta$ to be zero in (\ref{RIVF}). Following the
definition of the momentum mapping presented in (\ref{momlift}), the momentum map $\mathbf{J}_{L}^{G}$ is the projection from $%
(\Fg^\ast\dcc\Fh^\ast)\rtimes (G\bowtie H)$ to the first factor $\Fg^\ast$.
Hence the preimage of a regular value $\mu$ is isomorphic to $\Fh%
^\ast\rtimes (G\bowtie H)$. The isotropy subgroup $G_\mu$ acts on the preimage
as
\begin{equation*}
G_\mu\times (\Fh^\ast\rtimes (G\bowtie H)) \mapsto \Fh^\ast\rtimes (G\bowtie
H), \qquad (g,(\nu,g_2,h))\mapsto (g\overset{\ast }{\vartriangleright}\nu,gg_2,h).
\end{equation*}
The reduced symplectic manifold is then the quotient space $G_\mu\slash(\Fh^\ast\rtimes
(G\bowtie H))$ by the action of the isotropy subgroup $G_\mu$. This quotient space is
diffeomorphic to the product manifold $(G_\mu\slash\Fh^\ast)\times
\mathcal{O}_\mu \times H$. Here, $G_\mu\slash\Fh^\ast$ is the set of
right cosets whose elements are denoted by
\begin{equation*}
G_\mu\slash\Fh^\ast=\{_{G_\mu}[\nu]:=G_\mu\overset{\ast}{%
\vartriangleright}\nu \mid \nu \in \Fh^\ast \}.
\end{equation*}
We identify the tangent space $T_{_{G_\mu}[\nu]}(G_\mu\slash\Fh^\ast)$
with $\Fg_\mu\slash\Fh^\ast=\{_{\Fg_\mu}[\nu]:=\Fg_\mu\overset{\ast}{%
\vartriangleright}\nu \mid \nu \in \Fh^\ast \}$ where $\Fg_{\mu}$ is the
isotropy Lie subalgebra consisting of $\xi \in \Fg$ satisfying $%
ad^{\ast}_\xi \mu=0$, and a representative of $_{\Fg_\mu}[\nu]$
is given by $\xi\overset{\ast}{\vartriangleright}\nu = \hat\nu$ for $\xi \in \Fg_\mu $%
. Note that, the tangent space $T_\mu \mathcal{O}_\mu$ of the coadjoint orbit $\mathcal{O}_\mu$ consists of the
elements of the form $ad^{\ast}_\xi \mu$, with $\xi\in\Fg$.
Then the reduced symplectic two form $\Omega^{G\slash}$ takes the value
\begin{equation*}
\Omega^{G\slash}((\hat{\nu_1},ad^{\ast}_{\xi_1} \mu,TR_h \eta_1),(\hat{%
\nu_2},ad^{\ast}_{\xi_2} \mu,TR_h \eta_2))= \left\langle \hat{\nu_2}%
,\eta_1\right\rangle-\left\langle \hat{\nu_1},\eta_2\right\rangle+\left%
\langle (\mu ,\nu) ,[(\xi_1,\eta_1),(\xi_2,\eta_2)]\right\rangle
\end{equation*}
on
two vectors $(\hat{\nu_i},ad^{\ast}_{\xi_i} \mu,TR_h \eta_i)$, $i=1,2$, in the tangent
space at the point $\left(_{G_\mu}[\nu],\mu,h \right)$.
This symplectic structure is independent of $H$, hence, a symplectic
reduction and a Poisson reduction of $(G_\mu\slash\Fh^\ast)\times
\mathcal{O}_\mu \times H$ with respect to the right action of $H$ are also
possible. Note also that, in this case, the reduction by stages \cite%
{marsden2007hamiltonian} cannot be applicable since either $G$ or $H$ are
not necessarily normal subgroups of their matched pair.

\subsection{Legendre Transformation}

In our previous work \cite{essu2016LagMP}, we have been able to write the matched Euler-Lagrange
equation on the (left) trivialization of the tangent bundle $T(G\bowtie H)$ and Euler-Poincar\'{e}
equations on the matched pair Lie algebra $\Fg\bowtie \Fh$. In this subsection, we construct the Legendre transformation from the Lagrangian dynamics on $T(G\bowtie H)$ and $\Fg\bowtie \Fh$ to the Hamiltonian
dynamics on $T^\ast (G\bowtie H)$ and $\Fg^\ast \dcc \Fh^\ast$.

A Lagrangian density $\mathcal{L}=\mathcal{L}\left( g,h,\xi,\eta\right)$ on $(G\bowtie H)\ltimes(%
\Fg\bowtie\Fh) $ generates the matched Euler-Lagrange equations
\begin{align} \label{EL}
\begin{cases}
\displaystyle
\frac{d}{dt}\frac{\delta\mathcal{L}}{\delta\xi}  &  =T_{e_{G}}^{\ast}%
L_{g}\left(  \frac{\delta\mathcal{L}}{\delta g}\right)  \overset{\ast
}{\vartriangleleft}h+T_{e_{G}}^{\ast}\sigma_{h}\left(  \frac{\delta
\mathcal{L}}{\delta h}\right)  -ad_{\xi}^{\ast}\frac{\delta\mathcal{L}%
}{\delta\xi}+\frac{\delta\mathcal{L}}{\delta\xi}\overset{\ast}%
{\vartriangleleft}\eta+\mathfrak{a}_{\eta}^{\ast}\frac{\delta\mathcal{L}%
}{\delta\eta},\\
\displaystyle
\frac{d}{dt}\frac{\delta\mathcal{L}}{\delta\eta}  &  =T_{e_{H}}^{\ast}%
L_{h}\left(  \frac{\delta\mathcal{L}}{\delta h}\right)  -ad_{\eta}^{\ast
}\frac{\delta\mathcal{L}}{\delta\eta}-\xi\overset{\ast}{\vartriangleright
}\frac{\delta\mathcal{L}}{\delta\eta}-\mathfrak{b}_{\xi}^{\ast}\frac
{\delta\mathcal{L}}{\delta\xi}.
\end{cases}
\end{align}
When the Lagrangian is
independent of the group variables $g$ and $h$, that is, when there is a symmetry under the left action of the group $G \bowtie H$, then the matched pair Euler-Lagrange equations (\ref{EL})
reduce to the matched Euler-Poincar\'{e} equations
\begin{align} \label{mEP}
\begin{cases}
\displaystyle
\frac{d}{dt}\frac{\delta\mathfrak{L}}{\delta\xi}  &  =-ad_{\xi}^{\ast}%
\frac{\delta\mathfrak{L}}{\delta\xi}+\frac{\delta\mathfrak{L}}{\delta\xi
}\overset{\ast}{\vartriangleleft}\eta+\mathfrak{a}_{\eta}^{\ast}\frac
{\delta\mathfrak{L}}{\delta\eta},\\
\displaystyle
\frac{d}{dt}\frac{\delta\mathfrak{L}}{\delta\eta}  &  =-ad_{\eta}^{\ast}%
\frac{\delta\mathfrak{L}}{\delta\eta}-\xi\overset{\ast}{\vartriangleright
}\frac{\delta\mathfrak{L}}{\delta\eta}-\mathfrak{b}_{\xi}^{\ast}\frac
{\delta\mathfrak{L}}{\delta\xi}
\end{cases}
\end{align}
on the Lie algebra $\Fg \bowtie \Fh$.

We recall the diagram presented in (\ref{LLag-RHAm}) summarizing the Legendre transformation from the
left trivialization of the tangent bundle to the right trivialization of the cotangent bundle. From this, we first define a total energy
\begin{equation}
\mathcal{H}^L(g,h, \xi, \eta ,\mu^L, \nu^L)=\langle \mu^L, \xi \rangle + \langle \nu^L, \eta \rangle  - \mathcal{L}(g,h, \xi, \eta)
\end{equation}
on the left trivialization of the Pontryagin space $T(G \bowtie H)\times _ {G \bowtie H} T^\ast (G \bowtie H)$. When the Lagrangian is non degenerate, it is possible to solve the Lie algebra variables $\xi$ and $\eta$ in terms of the dual space elements $\mu^L,\nu^L$ from the equations
\begin{equation}
\mu^L=\frac{\delta\mathcal{L}}{\delta\xi}, \qquad \nu^L=\frac{\delta\mathcal{L}}{\delta\eta}.
\end{equation}
As such, the mapping
\begin{equation*}
(G \bowtie H)\ltimes (\Fg \bowtie \Fh)\rightarrow (G \bowtie H)\ltimes (\Fg^\ast \dcc \Fh^\ast)\qquad (g,h,\xi,\eta)\rightarrow (g,h,\frac{\delta \mathcal{L}}{\delta \xi},\frac{\delta \mathcal{L}}{\delta \eta})
\end{equation*}
becomes a diffeomorphism. In order to arrive at the Hamiltonian formulation on the right trivialization of the cotangent bundle, we employ the mapping
\begin{equation*}
\Phi:(G \bowtie H)\ltimes (\Fg^\ast \dcc \Fh^\ast)\rightarrow (\Fg^\ast \dcc \Fh^\ast)\rtimes (G \bowtie H),\qquad (g,h,\frac{\delta \mathcal{L}}{\delta \xi},\frac{\delta \mathcal{L}}{\delta \eta})
\rightarrow (Ad^{\ast}_{(g,h)^{-1}}(\frac{\delta\mathcal{L}}{\delta\xi},
\frac{\delta\mathcal{L}}{\delta\eta}),g,h),
\end{equation*}
 presented in (\ref{LR}). Here, $Ad^{\ast}$ is the coadjoint action given in (\ref{Ad-*}). If the Lagrangian is free from the group variables then this two step transformation takes the particular form
\begin{equation*}
\Fg \bowtie \Fh\rightarrow \Fg^\ast \dcc \Fh^\ast,\qquad (\xi,\eta)\rightarrow (Ad^{\ast}_{(g,h)^{-1}}(\frac{\delta\mathcal{L}}{\delta\xi},
\frac{\delta\mathcal{L}}{\delta\eta})).
\end{equation*}

For the degenerate cases, such a direct transformation is not possible. The Tulczyjew triplet is the geometric structure which enables the Legendre transformation for the degenerate systems \cite{tulczyjew1977legendre}. We also refer to \cite{esen2014tulczyjew,esen2015tulczyjew,zajkac2016tulczyjew} for the Lie groups.

\section{Example: Hamiltonian Dynamics on the group $SL(2,\mathbb{C}%
)$}

In this section we study the Hamiltonian dynamics on the dual space of the Lie algebra of $SL(2,\Cb)$. To this end we first recall the matched pair decomposition $SL(2,\Cb) = SU(2) \bowtie K$
into $SU(2)$ and its half-real form $K$ from \cite{Maji90-II}. It then follows that $\mathfrak{sl}(2,\Cb) =%
\mathfrak{su}(2)\bowtie \mathfrak{K}$, corresponding to the Iwasawa
decomposition of $\mathfrak{sl}(2,\mathbb{C})$ as a real Lie algebra, where $%
\mathfrak{su}(2)$ is the Lie algebra of $SU(2)$ and $\mathfrak{K}$ is the
Lie algebra of the group $K$.

We shall continue to follow the notation of \cite{essu2016LagMP}, namely,
\begin{align}
A & \in SU(2)\text{,}\qquad B\in K\text{,}\qquad\mathbf{X},\mathbf{X}_{1},%
\mathbf{X}_{2}\in\mathfrak{su}(2)\simeq\mathbb{R}^{3},\qquad \mathbf{Y},%
\mathbf{Y}_{1},\mathbf{Y}_{2}\in\mathfrak{K}\simeq\mathbb{R}^{3},  \notag \\
\mathbf{\Phi} & \in\mathfrak{su}^{\ast}(2)\simeq\mathbb{R}^{3}\text{,}\qquad%
\mathbf{\Psi}\in\mathfrak{K}^{\ast}\simeq\mathbb{R}^{3},\qquad \mathbf{k}%
=(0,0,1)\in\mathbb{R}^{3},  \label{Exnot}
\end{align}
for the generic elements of the groups $SU(2)$ and $K$, their Lie algebras $%
\mathfrak{su}(2)$ and $\mathfrak{K}$, and their linear duals $\mathfrak{su}%
(2)^{\ast} $ and $\mathfrak{K}^{\ast}$.

We recall that the group $SU(2)$ is the matrix Lie group
\begin{equation*}
SU(2)=\left\{
\begin{pmatrix}
\omega & \vartheta \\
-\bar{\vartheta} & \bar{\omega}%
\end{pmatrix}%
\in SL(2,\mathbb{C}):\left\vert \omega \right\vert ^{2}+\left\vert \vartheta
\right\vert ^{2}=1\right\},
\end{equation*}%
with the Lie algebra
\begin{equation*}
\mathfrak{su}(2)=\left\{ \frac{-\iota }{2}\left(
\begin{array}{cr}
t & r-\iota s \\
r+\iota s & -t%
\end{array}%
\right) :r,s,t\in \mathbb{R}\right\}
\end{equation*}%
of traceless skew-hermitian matrices. Following \cite{Maji90-II} we fix
\begin{equation*}
e_{1}=\left(
\begin{array}{cc}
0 & -\iota /2 \\
-\iota /2 & 0%
\end{array}%
\right) ,\,\,e_{2}=\left(
\begin{array}{cc}
0 & -1/2 \\
1/2 & 0%
\end{array}%
\right) ,\,\,e_{3}=\left(
\begin{array}{cc}
-\iota /2 & 0 \\
0 & \iota /2%
\end{array}%
\right)
\end{equation*}
as a basis of $\mathfrak{su}(2)$, via which we identify the Lie algebra $%
\mathfrak{su}(2) $ with the Lie algebra $\mathbb{R}^{3}$ by cross product,
\begin{equation*}
re_{1}+se_{2}+te_{3}\in \mathfrak{su}(2)\mapsto \mathbf{X}=\left(
r,s,t\right) \in \mathbb{R}^{3}.
\end{equation*}
Moreover, using the Euclidean dot product, we also identify the dual space $%
\mathfrak{su}(2)^{\ast }$ of $\mathfrak{su}(2)\simeq \mathbb{R}^{3}$ with $%
\mathbb{R}^{3}$. Then the coadjoint action of the Lie algebra $\mathfrak{su}%
(2)\simeq \mathbb{R}^{3}$ on $\mathfrak{su}(2)^{\ast }\simeq \mathbb{R}^{3}$
is given by
\begin{equation}
{\label{ad*1}} ad^{\ast }:\mathfrak{su}(2)\times \mathfrak{su}(2)^{\ast
}\rightarrow \mathfrak{su}(2)^{\ast },\text{ \ \ }(\mathbf{X},\mathbf{\Phi }%
)\mapsto ad_{\mathbf{X}}^{\ast }\mathbf{\Phi =X}\times \mathbf{\Phi },
\end{equation}
for $\mathbf{X}\in \mathfrak{su}(2)\simeq \mathbb{R}^{3}$ and $\mathbf{\Phi }%
\in \mathfrak{su}^{\ast }(2)\simeq \mathbb{R}^{3}$.

We recall also the subgroup $K\subseteq SL(2,\mathbb{C})$ which is given by
\begin{equation*}
K=\left\{ \frac{1}{\sqrt{1+c}}%
\begin{pmatrix}
1+c & 0 \\
a+ib & 1%
\end{pmatrix}
\in SL(2,\mathbb{C})\mid a,b\in\mathbb{R} \text{ and }c>-1\right\}.
\end{equation*}
The group $K$ can be identified with
\begin{equation*}
K=\left\{ (a,b,c)\in\mathbb{R}^{3} \mid a,b\in\mathbb{R}\text{ and }%
c>-1\right\},
\end{equation*}
where the group operation is
\begin{equation*}
(a_{1},b_{1},c_{1})%
\ast(a_{2},b_{2},c_{2})=(a_{1},b_{1},c_{1})(1+c_{2})+(a_{2},b_{2},c_{2}).
\end{equation*}
In this case $\mathfrak{K} = \mathbb{R}^{3}$ with the Lie bracket
\begin{equation}
\lbrack\mathbf{Y}_{1},\mathbf{Y}_{2}] = \mathbf{k}\times (\mathbf{Y}%
_{1}\times\mathbf{Y}_{2}),  \label{crossk}
\end{equation}
where $\mathbf{k}$ is the unit vector $(0,0,1)$. Then, the coadjoint action
of the Lie algebra $\mathfrak{K}\simeq\mathbb{R}^{3}$ on its dual space $%
\mathfrak{K}^{\ast} \simeq \mathbb{R}^3$ can be computed as
\begin{equation}
ad^{\ast}: \mathfrak{K}\times\mathfrak{K}^{\ast}\rightarrow\mathfrak{K}%
^{\ast },\text{ \ \ } (\mathbf{Y},\mathbf{\Psi})\mapsto ad_{\mathbf{Y}%
}^{\ast }\mathbf{\Psi}=\left( \mathbf{k}\cdot\mathbf{Y}\right) \mathbf{\Psi}%
-\left( \mathbf{\Psi}\cdot\mathbf{Y}\right) \mathbf{k}, {\label{ad*2}}
\end{equation}
for any $\mathbf{Y}\in\mathfrak{K}\simeq\mathbb{R}^{3}$, and any $\mathbf{%
\Psi}\in\mathfrak{K}^{\ast}\simeq\mathbb{R}^{3}$.

Recalling from \cite{essu2016LagMP}, the mutual actions of the Lie algebras $%
\mathfrak{su}(2)$ and $\mathfrak{K}$ are given by
\begin{eqnarray*}
\mathbf{\vartriangleright}&:&\mathfrak{K}\times\mathfrak{su}\left( 2\right)
\rightarrow\mathfrak{su}\left( 2\right) \text{, \ \ }({\mathbf{Y}}, {\mathbf{%
X}})\mapsto \mathbf{Y\vartriangleright X} := \mathbf{Y}\times(\mathbf{X}%
\times\mathbf{k}), \\
\vartriangleleft&:&\mathfrak{K}\times\mathfrak{su}\left( 2\right) \rightarrow%
\mathfrak{K,}\text{ \ \ }\left( \mathbf{Y},\mathbf{X}\right) \mapsto \mathbf{%
Y}\vartriangleleft\mathbf{X} := \mathbf{X}\times \mathbf{Y},
\end{eqnarray*}
for any $\mathbf{Y} \in \mathfrak{K}\cong\mathbb{R}^{3}$, and any $\mathbf{X}%
\in\mathfrak{su}\left( 2\right) \cong\mathbb{R}^{3}$. Accordingly, the dual
actions are given by
\begin{eqnarray}
\overset{\ast}{\vartriangleleft}&:&\mathfrak{su}\left( 2\right)^\ast \times
\mathfrak{K} \rightarrow\mathfrak{su}\left( 2\right)^\ast \text{, \ \ }(%
\mathbf{\Phi}, {\mathbf{Y}})\mapsto \mathbf{\Phi} \overset{\ast}{%
\vartriangleleft} \mathbf{Y} := \mathbf{\Phi} \times (\mathbf{k}\times
\mathbf{Y}),  \label{duals} \\
\overset{\ast}{\vartriangleright}&:&\mathfrak{su}\left( 2\right) \times
\mathfrak{K}^\ast \rightarrow \mathfrak{K}^\ast, \text{ \ \ } \left(\mathbf{X%
}, \mathbf{\Psi}\right) \mapsto \mathbf{X} \overset{\ast}{\vartriangleright}
\mathbf{\Psi} := \mathbf{\Psi} \times \mathbf{X}.  \label{duals2}
\end{eqnarray}

We recall the mappings (\ref{a}) and (\ref{etaonxi}), and write for the case
of $\mathfrak{sl}(2,\Cb)$ as follows
\begin{eqnarray*}
\mathfrak{a}_{\mathbf{Y}}&:&\mathfrak{su}\left( 2\right)\rightarrow
\mathfrak{K}, \qquad \mathbf{X}\mapsto \mathbf{X\times Y} \\
\mathfrak{b}_{\mathbf{X}}&:&\mathfrak{K}\rightarrow \mathfrak{su}\left(
2\right), \qquad {\mathbf{Y}}\mapsto \mathbf{Y}\times(\mathbf{X}\times%
\mathbf{k}),
\end{eqnarray*}
whose duals are given by
\begin{eqnarray}
\mathfrak{a}_{\mathbf{Y}}^{\ast}&:&\mathfrak{K}^{\ast}\rightarrow\mathfrak{su%
}\left( 2\right)^{\ast} \text{, \ \ }\mathbf{\Psi}\mapsto \mathfrak{a}_{%
\mathbf{Y}}^{\ast}({\mathbf{\Psi}}) = \mathbf{Y\times \Psi}  \label{a-b1} \\
\mathfrak{b}_{\mathbf{X}}^{\ast}&:&\mathfrak{su}\left( 2\right)^{\ast}
\rightarrow\mathfrak{K}^{\ast}\text{, \ \ }\mathbf{\Phi}\mapsto \mathfrak{b}%
_{\mathbf{X}}^{\ast}\mathbf{\Phi} = \left( \mathbf{\Phi\cdot k}\right)
\mathbf{X}-\left( \mathbf{\Phi\cdot X}\right) \mathbf{k}  \label{a-b2}
\end{eqnarray}
respectively.

Finally, the matched
Lie-Poisson equations (\ref{LPEgh})
for a Hamiltonian function $\mathcal{H}=\mathcal{H}(\mathbf{\Phi%
},\mathbf{\Psi})$ on $\Rb^3\dcc \Rb^3$ take the form of
\begin{equation}
\begin{cases}
\label{LPESL} \displaystyle\frac{d\mathbf{\Phi}}{dt} = \left( \frac{\delta%
\mathcal{H}}{\delta\mathbf{\Phi}}+ \frac{\delta\mathcal{H}}{\delta\mathbf{%
\Psi}}\times\mathbf{k} \right) \times \mathbf{\Phi}+\frac{\delta\mathcal{H}}{%
\delta\mathbf{\Psi}}\times \mathbf{\Psi}, \\
\displaystyle\frac{d\mathbf{\Psi}}{dt} = \left(\mathbf{k}\cdot\frac{\delta%
\mathcal{H}}{\delta\mathbf{\Psi}}\right)\mathbf{\Psi} -\left(\mathbf{\Psi}%
\cdot\frac{\delta\mathcal{H}}{\delta\mathbf{\Psi}} +\mathbf{\Phi}\cdot \frac{%
\delta\mathcal{H}}{\delta\mathbf{\Phi}}\right)\mathbf{k} +\mathbf{\Psi}%
\times \frac{\delta\mathcal{H}}{\delta\mathbf{\Phi}} +(\mathbf{\Phi}\cdot%
\mathbf{k})\frac{\delta\mathcal{H}}{\delta\mathbf{\Phi}}.%
\end{cases}%
\end{equation}

\section{Conclusion and Discussions}
In this paper, we have presented the canonical symplectic and Poisson structures on the (right) trivialization of the cotangent bundle $T^{\ast
}(G\bowtie H)$ of a matched pair Lie group $G\bowtie H$, see equations (\ref{SymT*GH}) and (\ref{PoissononT*GH}). We have written the Hamiltonian dynamics on the trivialized bundle, and applied the Poisson and the symplectic reductions under the group actions $G$ and $G\bowtie H$. The matched Lie-Poisson bracket (\ref{LiePoissonongh}) and the matched Lie-Poisson equations (\ref{LPEgh}) have been derived dualizing the Lie coalgebra structure on $\Fg^\ast\dcc\Fh^\ast$. As an example, we have written the Lie-Poisson equations on the dual space $\Rb^3\dcc \Rb^3$.

Inspired by the Kac decomposition \cite{Kac68,MoscRang09}, in a forthcoming paper, we are planning to
study the case of diffeomorphism groups.

\appendix

\section{Symplectic and Poisson Reductions}\label{symplec-Poisson-reduc}

Let $\mathcal{Q}$ be a smooth manifold, and $T^{\ast }%
\mathcal{Q}$ is its cotangent bundle. There exists a canonical one-form $\theta _{T^{\ast }\mathcal{Q}}$ on $T^{\ast }%
\mathcal{Q}$ whose exterior derivative is the canonical symplectic two-form $\omega _{T^{\ast }\mathcal{Q}}$. Let also, a Lie group $G$ acts on $T^\ast Q$ symplectically inducing an $Ad^{\ast }$-equivariant momentum map $%
\mathbf{J}_{T^{\ast }\mathcal{Q}}:T^{\ast }\mathcal{Q}\rightarrow \Fg^{\ast }
$ given by
\begin{equation}
\left\langle \mathbf{J}_{T^{\ast }\mathcal{Q}}\left( z\right) ,\xi
\right\rangle =\left\langle \theta _{T^{\ast }\mathcal{Q}},X_{\xi}^{T^{\ast }%
\mathcal{Q}}\right\rangle \left( z\right) ,  \label{mom}
\end{equation}%
where the pairing on left hand side is the one between $\Fg^{\ast }$ and $\Fg
$, and the pairing on right hand side is the one between $T_{z}^{\ast
}T^{\ast }\mathcal{Q}$ and $T_{z}T^{\ast }\mathcal{Q}$, \cite
{abraham1978foundations, arnol2013mathematical, holm2009geometric, libermann2012symplectic, MarsdenRatiu-book}. Here, $X_{\xi}^{T^{\ast }\mathcal{Q}}$ is the
infinitesimal generator of the action corresponding to $\xi \in \Fg$.

In
particular, let the left action $T^{\ast}\Phi_{g^{-1}}$ of $G$ on $%
T^{\ast }\mathcal{Q}$ be the cotangent lift of a left action $\Phi_g$ of $G$
on $\mathcal{Q}$ generated by the vector fields $X_{\xi}^{\mathcal{Q}}$. In
this case, the action is symplectic, generated by the cotangent lift $%
X_{\xi}^{T^{\ast }\mathcal{Q}}=(X_{\xi}^{\mathcal{Q}})^{c\ast}$ of $X_{\xi}^{%
\mathcal{Q}}$ satisfying the identity $T\pi_{\mathcal{Q}}\circ
X_{\xi}^{T^{\ast }\mathcal{Q}} =X_{\xi}^{\mathcal{Q}}\circ \pi_\mathcal{Q}$,
and the momentum map turns out to be
\begin{equation}
\left\langle \mathbf{J}_{T^{\ast }\mathcal{Q}}\left( z\right) ,\xi
\right\rangle =\left\langle z,X_{\xi}^{\mathcal{Q}}\right\rangle.
\label{momlift}
\end{equation}

Let $G_{\mu}$ be the isotropy group of $\mu \in \Fg^{\ast }$ under the
coadjoint action $Ad^{\ast }$ of $G$. Let also $\mu$ be a regular value of $%
\mathbf{J}$, that is, $\mathbf{J}^{-1}\left( \mu \right) $ is a submanifold
of $T^{\ast }\mathcal{Q}$, and it is an invariant set
for dynamics. As a result, $\mathbf{J}^{-1}\left( \mu \right)$ inherits the left action of the Lie group $G_{\mu }$.
If this action is free and proper, then $T^{\ast }\mathcal{Q}^{\left. G\right\slash }:=G_{\mu }\slash
\mathbf{J}_{T^{\ast }\mathcal{Q}}^{-1}\left( \mu \right)$ is a manifold,
called the reduced phase space. Let $\imath_{\mu}:\mathbf{J}%
^{-1}\left( \mu \right) \rightarrow T^{\ast }\mathcal{Q}$ be the natural
injection, and $\chi _{\mu }:\mathbf{J}%
^{-1}\left( \mu \right) \rightarrow G_{\mu }\slash \mathbf{J}_{T^{\ast }%
\mathcal{Q}}^{-1}\left( \mu \right)$ be the projection. The reduced manifold $G_{\mu }\slash \mathbf{J}_{T^{\ast }%
\mathcal{Q}}^{-1}\left( \mu \right)$ has a unique symplectic structure $%
\Omega_{{\mu }}$ satisfying $\chi_{\mu }^{\ast }\Omega _{{\mu }}=\imath_{\mu
}\Omega _{T^{\ast }\mathcal{Q}}$. In short, the following diagram commutes (known as the symplectic
reduction theorem \cite{marsden1974reduction,meyer1973symmetries}):
\begin{equation}
\xymatrix{\mathbf{J}_{T^{\ast }\mathcal{Q}}^{-1}\left( \mu
\right)\ar[rrr]^{\imath _{\mu}}\ar[dd]_{\chi_{\mu}}&&&T^{\ast }\mathcal{Q}
\ar[ddlll]^{p_{\mu}} \\\\T^{\ast }\mathcal{Q}^{\left. G\right\slash
}:=G_{\mu }\slash \mathbf{J}_{T^{\ast }\mathcal{Q}}^{-1}\left( \mu
\right)}  \label{srd}
\end{equation}
where $p_{\mu }$ is the projection from $T^{\ast }\mathcal{Q}$ to the
reduced space $T^{\ast }\mathcal{Q}^{\left. G\right\slash }$.

In order to evaluate the symplectic form, let $\left[ z\right] \in
T^{\ast }\mathcal{Q}^{\left. G\right\slash }$ and $X_{\left[ z\right]
},Y_{\left[ z\right] }\in T_{\left[ z\right] }T^{\ast }\mathcal{Q}^{\left.
G\right\slash }$. Then, for $z\in p _{\mu
}^{-1}\left( \left[ z\right] \right)$, we have $T_{z}p _{\mu }\left( X_{z}\right) =X_{\left[ z \right] }$ and $T_{z}p
_{\mu }\left( Y_{z}\right) =Y_{\left[ z\right] }$, and the value of symplectic form is given by
\begin{equation}
\Omega _{\mu }\left( \left[ z\right] \right) \left( X_{\left[ z\right] },Y_{%
\left[ z\right] }\right) =\left. \Omega \left( z\right) \right\vert _{%
\mathbf{J}^{-1}\left( \mu \right) }\left( X_{z},Y_{z}\right).
\end{equation}

Given a $G$-invariant Hamiltonian
function $h,$ the reduced Hamiltonian function is $h_{\mu }=h\circ p _{\mu }$. Then the corresponding Hamiltonian vector fields $X_{h}$ and $X_{h_{\mu }}$ are $p_{\mu }$-related, and the trajectories of $X_{h}$ project onto those of $X_{h_{\mu }}$.

Let $\left( \mathcal{P},\left\{ \text{ },\text{ }\right\} _{\mathcal{P}}\right)$ be a
Poisson manifold with a free, proper, and canonical action of a Lie group $G$. There is then a Poisson structure on the quotient
manifold $\mathcal{P}/G$ which makes the projection $\pi :\mathcal{P}%
\rightarrow \mathcal{P}/G$ a Poisson map. For the $G$-invariant
functions $f$ and $h$ on $\mathcal{P}$, the Poisson structure on the quotient is given by
\begin{equation}  \label{PR}
\left\{ f\circ \pi ,h\circ \pi \right\} _{\mathcal{P}/G}=\left\{ f,h\right\}
_{\mathcal{P}}\circ \pi .
\end{equation}%
This procedure is called the Poisson reduction \cite{marsden1986reduction}.

\section{Hamiltonian Dynamics on Lie groups}\label{Hamilton-dynam-matched}

The cotangent lifts of the left and right translations on $G$ result with
the actions of $G$ on its cotangent bundle $T^{\ast}G$ given by
\begin{eqnarray}
G\times T^{\ast }G\mapsto T^{\ast }G,\qquad (g_{1},\alpha_{g_2}) \mapsto
T_{g_{1}g_{2}}^{\ast }L_{g_{1}^{-1}}\alpha_{g_2},  \label{GonT*G} \\
T^{\ast }G\times G\mapsto T^{\ast }G,\qquad (\alpha_{g_1},g_{2}) \mapsto
T_{g_{1}g_{2}}^{\ast }R_{g_{2}^{-1}}\alpha_{g_1},  \label{T*GonG}
\end{eqnarray}
respectively. On the other hand, the cotangent bundle $T^{\ast }G$ admits global trivializations
\begin{eqnarray}
tr_{T^{\ast }G}^{R}:T^{\ast }G\mapsto \Fg^{\ast }\rtimes G,\qquad \alpha
_{g}\mapsto (\mu =T_{e}^{\ast }R_{g}\alpha _{g},g),  \label{rt*G} \\
tr_{T^{\ast }G}^{L}:T^{\ast }G\mapsto G\ltimes\Fg^{\ast },\qquad \alpha
_{g}\mapsto (\mu =T_{e}^{\ast }L_{g}\alpha _{g},g),  \label{lt*G}
\end{eqnarray}
by which the semi-direct product Lie group structures
\begin{eqnarray}
(\mu _{1},g_{1})(\mu _{2},g_{2})=(\mu _{1}+Ad_{g_{1}}^{\ast }\mu
_{2},g_{1}g_{2}),  \label{GrT*Gright} \\
(g_{1},\mu _{1})(g_{2},\mu _{2})=(g_{1}g_{2}, Ad_{g_{2}^{-1}}^{\ast }\mu
_{1}+\mu _{2}),  \label{GrT*Gleft}
\end{eqnarray}%
on $\Fg^\ast \rtimes G$ and $G\ltimes \Fg^\ast$, respectively, are pulled back to $T^{\ast}G$. As such, the cotangent bundle carries a group structure as well, \cite%
{lichnerowicz1986characterization}. The Lie algebra
of the group $\mathfrak{g}^\ast\rtimes G$ is the semi-direct sum Lie algebra $\mathfrak{g}^\ast\rtimes \mathfrak{g}$ equipped with the Lie bracket
\begin{eqnarray*}
[(\hat{\mu}_{1},\xi _{1}),(\hat{\mu}_{2},\xi _{2})] = \left( ad_{\xi _{1}}^{\ast }\hat{\mu}_{2}-ad_{\xi _{2}}^{\ast }\hat{\mu}%
_{1},[\xi _{1},\xi _{2}]\right),
\end{eqnarray*}
for two generic elements $(\hat{\mu}_{1},\xi _{1}),(\hat{\mu}_{2},\xi _{2})\in\mathfrak{g}^\ast\rtimes \mathfrak{g}$. Similarly, the Lie algebra $\Fg \ltimes \Fg^\ast$ admits the semi-direct sum Lie algebra structure given by
\begin{eqnarray*}
[(\xi _{1}, \hat{\mu}_{1}),(\xi _{2},\hat{\mu}_{2})] = \left([\xi _{1},\xi _{2}], ad_{\xi _{1}}^{\ast }\hat{\mu}_{2}-ad_{\xi _{2}}^{\ast }\hat{\mu}%
_{1}\right),
\end{eqnarray*}

Accordingly, the lifted left and the right actions (\ref{GonT*G}) and (\ref{T*GonG}) on the (right) trivialized - by (\ref{rt*G}) - cotangent bundle
are given by
\begin{eqnarray}
G\times (\Fg^{\ast }\rtimes G)\mapsto \Fg^{\ast }\rtimes G,\qquad
(g_{1},(g_{2},\mu _{2}))&=&(Ad^{\ast}_{g_{1}}\mu _{2},g_{1}g_{2}),
\label{GontrT*Gleft} \\
(\Fg^{\ast }\rtimes G)\times G\mapsto \Fg^{\ast }\rtimes G,\qquad (\mu
_{1},g_{1})(g_{2})&=&(\mu _{1},g_{1}g_{2}),  \label{GontrT*Gright}
\end{eqnarray}
or directly by setting $\mu_1$ (and respectively $\mu_2$) to be zero in the group operation (\ref{GrT*Gright}).

Being a cotangent bundle, $T^{\ast }G$ is a symplectic manifold carrying an
exact symplectic two-form. Using the (right) trivialization map \eqref{rt*G}, the canonical one-form $\theta _{T^{\ast }G}$, and the symplectic
two-from $\Omega _{T^{\ast }G}$ on $T^{\ast }G$, can be pushed forward to $\Fg^{\ast }\rtimes G$.
As such, there appears a one-form $\theta _{\Fg^{\ast }\rtimes G}$, and a
symplectic two-form $\Omega _{\Fg^{\ast }\rtimes G}$ on the semidirect
product Lie group $\Fg^{\ast }\rtimes G$.

Recalling that the right invariant vector fields $%
X_{\left( \hat{\mu},\xi \right) }^{\Fg^{\ast }\rtimes G}$ on $\Fg^{\ast
}\rtimes G$ are determined uniquely by an element $\left( \hat{\mu},\xi
\right) $ in the Lie algebra $\Fg^{\ast }\rtimes \Fg$ of $\Fg^{\ast }\rtimes
G$ as
\begin{equation}
{\label{livfG}}X_{\left( \hat{\mu},\xi \right) }^{\Fg^{\ast }\rtimes G}(\mu
,g)=\left( \hat{\mu}+ad_{\xi }^{\ast }\mu ,TR_{g}{\xi }\right),
\end{equation}%
the values of the one-form $\theta _{\Fg^{\ast }\rtimes G}$, and the symplectic
two-from $\Omega _{\Fg^{\ast }\rtimes G}$, on the right invariant vector fields $%
{}^{R}X_{\left( \hat{\mu},\xi \right) }^{\mathfrak{g}^{\ast }\rtimes G}$ are given by
\begin{eqnarray}
\left\langle \theta _{\Fg^{\ast }\rtimes G},X_{\left( \hat{\mu},\xi \right)
}^{\Fg^{\ast }\rtimes G}\right\rangle \left( \mu, g\right) &=&\left\langle
\mu ,\xi \right\rangle,  \label{1formT*G} \\
\left\langle \Omega _{\Fg^{\ast }\rtimes G};\left( X_{\left(\hat{\mu}%
_{1},\xi _{1}\right) }^{\Fg^{\ast }\rtimes G},X_{\left( \hat{\mu}_{2},\xi
_{2}\right) }^{\Fg^{\ast }\rtimes G}\right) \right\rangle \left( \mu,g
\right) &=&\left\langle \hat{\mu}_{2},\xi _{1}\right\rangle-\left\langle
\hat{\mu}_{1},\xi _{2}\right\rangle +\left\langle \mu ,[\xi _{1},\xi
_{2}]\right\rangle ,  \label{2formT*G}
\end{eqnarray}%
where the bracket in the right hand side of the second line
is the Lie algebra bracket on $\Fg$, see \cite{abraham1978foundations,
AlekGrabMarmMich94, manga2015geometry}.

It is thus possible to derive the Hamilton's equations on the trivialized cotangent bundle. The Hamilton's equations are defined
by $i_{X_{\mathcal{H}}}\Omega =d\mathcal{H}$. Hence, for a
Hamiltonian function $\mathcal{H}(\mu,g)$ on $\Fg^{\ast }\rtimes G$, the
Hamilton's equations turn out to be
\begin{equation}  \label{HameqonT*G}
\dot{\mu}=-T^{\ast}R_g(\frac{\delta \mathcal{H}}{\delta g})+ad^{\ast}_{%
\frac{\delta \mathcal{H}}{\delta \mu}}\mu, \qquad \dot{g}=TR_{g}(\frac{%
\delta \mathcal{H}}{\delta \mu}).
\end{equation}
The canonical Poisson bracket of two functionals $\mathcal{H}$ and $%
\mathcal{F}$ on $\Fg^{\ast }\rtimes G$ is similarly given by
\begin{equation}  \label{PoissononT*G}
\left\{ \mathcal{H},\mathcal{F}\right\} _{\Fg^{\ast }\rtimes
G}(\mu,g)=\left\langle T^{\ast}R_{g}(\frac{\delta \mathcal{H}}{\delta g}),%
\frac{\delta \mathcal{F}}{\delta \mu}\right\rangle -\left\langle
T^{\ast}R_{g}(\frac{\delta \mathcal{F}}{\delta g}),\frac{\delta \mathcal{H}%
}{\delta \mu}\right\rangle +\left\langle \mu ,[\frac{\delta \mathcal{H}}{%
\delta \mu},\frac{\delta \mathcal{F}}{\delta \mu}]\right\rangle.
\end{equation}

\section{Reductions of Hamiltonian Dynamics on Lie Groups}\label{reduc-Hamilton-dynam}

The infinitesimal generators of the left action of $G$ on $\Fg^{\ast
}\rtimes G$ are the right invariant vector fields on $\Fg^{\ast }\rtimes G$, and
they are of the form ${}^{R}X_{\left(0,\xi \right) }^{\Fg^{\ast }\rtimes G}$,
simply by setting $\hat{\mu}=0$ in (\ref{livfG}). It then follows from (\ref{mom}) and (\ref{1formT*G})
that the momentum mapping $\mathbf{J}_L$ of the left action of $G$ on $\Fg%
^{\ast }\rtimes G$ becomes simply the projection onto the first factor, that is, $%
\mathbf{J}_L:(\mu,g)\mapsto\mu$. In other words, the inverse image $%
\mathbf{J}_L^{-1}(\mu)\in \Fg^\ast\rtimes G$ of any $\mu \in \Fg^\ast$ can be identified with the group $G$.
In this case, the isotropy subgroup $G_{\mu }$ of the coadjoint action is given by $\left\{ g\in
G:Ad_{g}^{R\ast }\mu =\mu \right\}$. Hence, each quotient space $G_{\mu
}\slash \mathbf{J}_{L}^{-1}\left( \mu \right)$ is isomorphic to the
coadjoint orbit $\mathcal{O}_\mu=\left\{ Ad_{g}^{\ast }\mu \in \Fg^{\ast
}:g\in G\right\}$ through the point $\mu \in \Fg^{\ast }$. More explicitly, the
identification $G_{\mu }\slash G \cong \mathcal{O}_{\mu }$
is given by $[g]_{\mu }\mapsto Ad_{g}^{\ast }\mu $. As a result, the
reduction diagram (\ref{srd}) turns out to be
\begin{equation}  \label{SRT*G}
\xymatrix{\mathbf{J}_{L}^{-1}\left( \mu \right) =G \ar[rr]^{\imath _{\mu}}
\ar[dd]_{\chi _{\mu}} && \Fg^{\ast }\rtimes G \ar[ddll]^{p_{\mu}}\\\\ G_{\mu
}\slash \mathbf{J}_{L}^{-1}\left( \mu \right)=\mathcal{O}_{\mu }}.
\end{equation}
The reduced space $\mathcal{O}_{\mu }$ is a symplectic manifold with the
reduced symplectic Kostant-Krillov-Souriau two-form $\Omega _{\Fg^{\ast
}\rtimes G}^{G\slash }$ whose values on the vectors $X^{\mathcal{O%
}_{\mu }}_{\xi}(\mu)=ad_{\xi }^{\ast }\mu$ are given by
\begin{equation}
\left\langle \Omega _{\Fg^{\ast }\rtimes G}^{G\slash };(X^{\mathcal{O}%
_{\mu }}_{\xi_1},X^{\mathcal{O}_{\mu }}_{\xi_2}) \right\rangle \left( \mu
\right) =\left\langle \mu ,\left[ \xi_1,\xi_2 \right] \right\rangle,
\label{KKS}
\end{equation}%
which is the well-known form of the coadjoint orbit symplectic structure.

The group $G$ acts on the symplectic manifold $(\Fg^{\ast }\rtimes G,\Omega
_{\Fg^{\ast }\rtimes G})$ from the right as well,  (\ref{GontrT*Gright}). A
Hamiltonian is invariant under this action if it is independent of the group
variable. According to the Poisson reduction formula (\ref{PR}), for
two such functions $\mathcal{H}=\mathcal{H}(\mu)$ and $\mathcal{F}=%
\mathcal{F}(\mu)$, the canonical Poisson bracket (\ref{PoissononT*G})
reduces to the Lie-Poisson bracket on the dual space $\Fg^{\ast }$. Namely,
\begin{equation}  \label{LPB}
\left\{ \mathcal{H},\mathcal{F}\right\}_{\Fg^{\ast }}(\mu)=\left\langle
\mu ,[\frac{\delta \mathcal{H}}{\delta \mu},\frac{\delta \mathcal{F}}{%
\delta \mu}]\right\rangle,
\end{equation}%
where $\frac{\delta \mathcal{H}}{\delta \mu},\frac{\delta \mathcal{F}}{%
\delta \mu}\in \mathfrak{g}$, and the bracket on the right hand side is the
Lie algebra bracket on $\mathfrak{g}$ assuming the reflexivity .

%The Poisson mapping relating
%these Poisson structures is the momentum mapping $\mathbf{J}_L$
%induced by the left action of $G$ on $\Fg^{\ast }\rtimes G$. This is the
%manifestation of the existence of a dual pair (c.f. see following appendix).
In the reduced picture, the dynamics is governed by the Lie-Poisson equations,
which can be obtained by setting $\frac{%
\delta \mathcal{H}}{\delta g}$ to be zero in (\ref{HameqonT*G}), see \cite%
{marsden1991symplectic},
\begin{equation}  \label{LPong*}
\dot{\mu}=ad^{\ast}_{\frac{\delta \mathcal{H}}{\delta \mu}}\mu
\end{equation}
for a Hamiltonian function $\mathcal{H}=\mathcal{H}(\mu)$.

Finally, the symplectic leaves of $\Fg%
^{\ast }$ equipped with the Lie-Poisson bracket (\ref{LPB}) are the
coadjoint orbits $\mathcal{O}_{\mu }$ of the symplectic reduction (\ref{SRT*G}),
using (\ref{KKS}), \cite{kostant1970orbits,kyrillov1974elements,souriau1966structure}.

\section{The Legendre Transformation} \label{LT}

In order to connect the Hamiltonian dynamics presented in (\ref{HameqonT*G}%
) to the Lagrangian dynamics (we recall below), the Hamilton's equations are written on the left trivialization
$G\ltimes \Fg^{\ast}$ of the cotangent bundle $T^{\ast}G$, by (\ref{lt*G}). To this end, it suffices to consider the mapping
\begin{equation}\label{LR}
\Phi:G\ltimes \Fg^{\ast}\ra\Fg^{\ast}\rtimes G \qquad
(g,\mu^{L})\rightarrow(Ad^{\ast}_g\mu^{L},g).
\end{equation}
Then, for any
real-valued function $\mathcal{H}$ on $\Fg^{\ast}\rtimes G$, there is a
real-valued function $\mathcal{H}^{L}=\mathcal{H}\circ \Phi$ on $G\ltimes %
\Fg^{\ast}$. Substitutions of the variations
\begin{equation*}
\frac{\delta\mathcal{H}}{\delta g}=\frac{\delta\mathcal{H}^{L}}{\delta g}%
+T^{\ast}R_{g^{-1}}\circ Ad^{\ast}_g \circ ad^{\ast}_\frac{\delta\mathcal{H}%
^L}{\delta \mu^{L}}\mu^{L}, \qquad \frac{\delta\mathcal{H}}{\delta \mu}=Ad_g%
\frac{\delta\mathcal{H}^L}{\delta\mu^{L}}
\end{equation*}
into the Hamilton's equations (\ref{HameqonT*G}), and a direct calculation,
result with the left version of the Hamilton's equations
\begin{equation}  \label{HameqonT*GL}
\dot{\mu}=-T^{\ast}L_g(\frac{\delta \mathcal{H}^{L}}{\delta g})-ad^{\ast}_{%
\frac{\delta \mathcal{H}^{L}}{\delta \mu^{L}}}\mu^{L}, \qquad \dot{g}%
=TL_{g}(\frac{\delta \mathcal{H}^{L}}{\delta \mu^{L}})
\end{equation}
on $G\ltimes \Fg^{\ast}$. If, in particular, the Hamiltonian is free from the
group variable, then the last equations descend to the Lie-Poisson equations with a negative sign
\begin{equation}  \label{LP-L}
\dot{\mu}=-ad^{\ast}_{\frac{\delta \mathcal{H}^{L}}{\delta \mu^{L}}%
}\mu^{L}, \qquad \dot{g}=TL_{g}(\frac{\delta \mathcal{H}^{L}}{\delta \mu^{L}%
}).
\end{equation}
On the other hand, the trivialized Euler-Lagrange equations generated by a Lagrangian $\mathcal{L}=\mathcal{L}(g,\xi)$ on the left
trivialization $G\ltimes\Fg$ of the tangent bundle $TG$ are
\begin{equation}
\frac{d}{dt}\frac{\delta\mathcal{L}}{\delta\xi}=T_{e}^{\ast}L_{g}\frac {%
\delta\mathcal{L}}{\delta g}-ad_{\xi}^{\ast}\frac{\delta\mathcal{L}}{%
\delta\xi},   \label{preeulerlagrange}
\end{equation}
\cite{essu2016LagMP}. For $\mu^L=%
\frac{\delta\mathcal{L}}{\delta\xi}$, the function
\begin{equation}
\mathcal{H}^L(g,\xi,\mu^L)=\langle \mu^L, \xi \rangle - \mathcal{L}(g,\xi)
\end{equation}
on the left trivialization $G\ltimes(\Fg\times\Fg^{\ast})$ of the Whitney
product $T^\ast G \times_G TG$ is known as the total energy function. If the Lagrangian is
nondegenerate, then $\xi$ can be solved as a function of $\mu^L$ from the equality
$\mu^L=\frac{\delta\mathcal{L}}{\delta\xi}$ to obtain a nondegenerate
Hamiltonian function on $G\ltimes\Fg^{\ast}$. In this case, the Legendre transformation converts the Hamilton's
equations (\ref{HameqonT*GL}) to the trivialized Euler-Lagrange equations (%
\ref{preeulerlagrange}). It is immediate now to write the Legendre
transformation for the left invariant formulations that is, when the
Lagrangian (resp. Hamiltonian) function is free from group variable. We
refer \cite{esen2014tulczyjew,esen2015tulczyjew} for the cases where the
dynamics is not necessarily nondegenerate.

Finally, the Legendre transformation from the left Lagrangian dynamics to the right Hamilton's
equations is given by the commutative diagram
\begin{equation}  \label{LLag-RHAm}
\xymatrix{G\ltimes \Fg \ar[rr]^{\Phi\circ \mathbb{F}\mathcal{L}}
\ar[dr]_{\mathcal{L}} && \Fg^{\ast }\rtimes G \ar[dl]^{\mathcal{H}}\\
&\mathbb{R}}
\end{equation}
where $\mathbb{F}\mathcal{L}$ stands for the fiber derivative $\frac{\delta%
\mathcal{L}}{\delta\xi}$ of the Lagrangian, and
$\Phi\circ \mathbb{F}\mathcal{L}: (g,\xi)\mapsto (Ad^{\ast}_{g^{-1}}%
\frac{\delta\mathcal{L}}{\delta\xi},g)$ is the Legendre transformation.

\section{Lie coalgebras} \label{LCo}

\label{CBMPLG} A vector space $\mathfrak{G}$ is called a Lie coalgebra if
there is a linear map $\d:\mathfrak{G} \ra \mathfrak{G} \ot \mathfrak{G}$
denoted by $\d(\mu) = \mu\nsb{1} \ot \mu\nsb{2}$ such that
\begin{align*}
\mu\nsb{1} \ot \mu\nsb{2} = - \mu\nsb{2} \ot \mu\nsb{1}, \qquad \mu\nsb{1} %
\ot \mu\nsb{2}\nsb{1} \ot \mu\nsb{2}\nsb{2} + \mu\nsb{2}\nsb{1} \ot \mu%
\nsb{2}\nsb{2} \ot \mu\nsb{1} + \mu\nsb{2}\nsb{2} \ot \mu\nsb{1} \ot \mu%
\nsb{2}\nsb{1} = 0,
\end{align*}
see \cite{Majid-book,Mich80}. The linear algebraic dual $%
\mathfrak{G}^\ast$ of $\mathfrak{G}$ is a Lie algebra. Inversely, dual $\Fg%
^{\ast}$ of the Lie algebra $\Fg$ is a Lie coalgebra with the cobracket
given by
\begin{equation*}
\langle \d(\mu),(\xi_1,\xi_2)\rangle = \langle \mu,[\xi_1,\xi_2]\rangle.
\end{equation*}
Let $(\mathfrak{G},\d)$ be a Lie coalgebra and $V$ be an arbitrary vector
space. $V$ is called a (right) $\mathfrak{G}$-comodule if there exists a map,
called the right coaction of a Lie coalgebra, $\nb:V \ra V\ot \mathfrak{G}$
given by $\nb(v) = v\nsb{0} \ot v\nsb{1}$ such that
\begin{equation*}
v\nsb{0} \ot v\nsb{1}\nsb{1} \ot v\nsb{1}\nsb{2} = v\nsb{0}\nsb{0} \ot v%
\nsb{0}\nsb{1} \ot v\nsb{1} - v\nsb{0}\nsb{0} \ot v\nsb{1} \ot v\nsb{0}%
\nsb{1}
\end{equation*}
for any $v\in V$. Similarly, $V$ is called a left $\mathfrak{G}$-comodule, if there is a mapping, called the left coaction
of a Lie coalgebra, $\nb:V \ra \mathfrak{G}\ot V$, $\nb(v) = v\nsb{-1} \ot v%
\nsb{0}$, such that
\begin{equation*}
v\nsb{-1}\nsb{1} \ot v\nsb{-1}\nsb{2} \ot v\nsb{0} = v\nsb{-1} \ot v\nsb{0}%
\nsb{-1} \ot v\nsb{0}\nsb{0} - v\nsb{0}\nsb{-1} \ot v\nsb{-1} \ot v\nsb{0}%
\nsb{0}
\end{equation*}
for any $v\in V$.

\bibliographystyle{amsplain}
\bibliography{references}

\end{document}